\documentclass[11pt]{article}

\usepackage[T1]{fontenc}
\usepackage[utf8]{inputenc}
\usepackage[english]{babel}
\usepackage{color}
\usepackage{wrapfig}
\usepackage{placeins}
\usepackage{subfigure}
\usepackage{tabu}
\usepackage{fullpage}
\usepackage[squaren]{SIunits}
\usepackage{graphicx}
\usepackage{natbib}
\usepackage{tikz}
\usepackage{pgfplots}
\usepackage{epstopdf}
\usepackage{epsfig}
\usepackage{pifont}
\newcommand{\cmark}{\ding{51}}
\newcommand{\xmark}{\ding{55}}
\usepackage[colorlinks=true,bookmarks=false,linkcolor=blue,urlcolor=blue,citecolor=blue,breaklinks=true]{hyperref}
\usepackage{breakcites}
\usepackage{tikz}
\usepackage{tikz-qtree}
\usepackage{eurosym}
\usepackage{amsmath}
\numberwithin{equation}{section}
\usepackage{amssymb}
\usepackage{amsthm}
\usepackage{mathrsfs}
\usepackage{blkarray}
\usepackage{dsfont}
\usepackage{comment}
\usepackage{rotating}

\newcommand{\R}{\mathbb{R}}
\newcommand{\Mcal}{\mathcal{M}}
\newcommand{\M}{\mathcal{M}}
\newcommand{\calM}{\mathcal{M}}

\newcommand{\Acal}{\mathcal{A}}

\newcommand{\bigO}{\mathcal{O}}

\definecolor{purple}{rgb}{0.74, 0.2, 0.64}

\newcommand{\Lcal}{\mathcal{L}}
\newcommand{\Ccal}{\mathcal{C}}
\newcommand{\calC}{\mathcal{C}}
\newcommand{\Dcal}{\mathcal{D}}
\newcommand{\calD}{\mathcal{D}}
\newcommand{\D}{\mathrm{D}}
\newcommand{\ddt}{\frac{\mathrm{d}}{\mathrm{d}t}}
\newcommand{\Ecal}{\mathcal{E}}
\newcommand{\calE}{\mathcal{E}}
\newcommand{\calL}{\mathcal{L}}

\newcommand{\txt}{\text}
\newcommand{\transpose}{^\top}

\newcommand{\St}{\mathrm{St}}
\newcommand{\grad}{\mathrm{grad}}
\newcommand{\Proj}{\mathrm{Proj}}

\newcommand{\Hess}{\mathrm{Hess}}
\newcommand{\hess}{\nabla^2}

\newcommand{\T}{\mathrm{T}}
\newcommand{\Nrm}{\mathrm{N}}

\newcommand{\Sym}{\mathrm{Sym}}
\newcommand{\Rm}{\mathbb{R}^m}
\newcommand{\Rn}{\mathbb{R}^n}
\newcommand{\rank}{\operatorname{rank}}
\newcommand{\spann}{\operatorname{span}}
\newcommand{\nl}{\newline}
\newcommand{\inner}[2]{\left\langle{#1},{#2}\right\rangle}
\newcommand{\Id}{\mathrm{Id}}
\newcommand{\I}{\mathrm{I}}
\newcommand{\sigmamin}{\sigma_\mathrm{min}}
\newcommand{\sigmamax}{\sigma_\mathrm{max}}
\newcommand{\sigmabar}{\underline{\sigma}}
\newcommand{\Clambdabar}{\overline{C_{\lambda}}}
\newcommand{\betabar}{\overline{\beta}}
\newcommand{\lambdamin}{\lambda_\mathrm{min}}
\newcommand{\LP}{\mathrm{LP}}
\newcommand{\tmax}{t_{\mathrm{max}}}

\usepackage{bm}
\DeclareMathOperator*{\argmin}{\arg\!\min}

\usepackage{algorithm}
\usepackage{algpseudocode}
\usepackage{todonotes}

\newcommand{\aref}[1]{\hyperref[#1]{A\ref{#1}}}
\newcommand{\norm}[1]{\left\|#1\right\|}
\newcommand{\fronorm}[1]{\left\|#1\right\|_\mathrm{F}}
\newcommand{\opnorm}[1]{\left\|{#1}\right\|_\mathrm{op}}

\usepackage{apxproof}
\newtheorem{theorem}			     {Theorem}	[section]
\newtheorem{corollary}	  [theorem]	 {Corollary}	
\newtheorem{lemma}	      [theorem]  {Lemma}		
\newtheorem{definition}	         {Definition}[section]
\newtheorem{assumption} {A\ignorespaces}
\newtheorem{remark}{Remark}[section]
\newtheorem*{example*}{Example}
\newtheoremrep{proposition} [theorem] {Proposition}

\usepackage{listings}
\usepackage{textcomp}
\definecolor{listinggray}{gray}{0.9}
\definecolor{lbcolor}{rgb}{0.9,0.9,0.9}
\usepackage{mathtools}
\mathtoolsset{showonlyrefs=false}

\usepackage{authblk}

\excludecomment{arxiv} 

\includecomment{journal}

\begin{document}
\title{Computing second-order points under equality constraints: revisiting Fletcher's augmented Lagrangian}

\author[1]{Florentin Goyens*}
\author[ ]{Armin Eftekhari}
\author[2]{Nicolas Boumal}
\affil[1]{\textsc{LAMSADE}, Université Paris Dauphine-PSL\\
Paris, France}
\affil[2]{Ecole Polytechnique Fédérale de Lausanne (EPFL), Insitute of Mathematics, Switzerland}
\affil[*]{Corresponding author: \url{goyensflorentin@gmail.com} }

\date{\today}

\maketitle

\begin{abstract}
We address the problem of minimizing a smooth function under smooth equality constraints. Under regularity assumptions on these constraints, we propose a notion of approximate first- and second-order critical point which relies on the geometric formalism of Riemannian optimization. Using a smooth exact penalty function known as Fletcher's augmented Lagrangian, we propose an algorithm to minimize the penalized cost function which reaches $\varepsilon$-approximate second-order critical points of the original optimization problem in at most $\bigO(\varepsilon^{-3})$ iterations. This improves on current best theoretical bounds.
Along the way, we show new properties of Fletcher's augmented Lagrangian, which may be of independent interest. \newline
\end{abstract}

\section{Introduction}
\label{sec:intro}
Working over a Euclidean space $\calE$ with inner product $\left\langle \cdot, \cdot\right\rangle$ and associated norm $\norm{\cdot}$, we consider the constrained optimization problem
\begin{equation}
\underset{x}{ \min} f(x) \textrm{ subject to } h(x)=0,
\label{eq:P}
\tag{P}
\end{equation}
where $f \colon \calE \to \R$ and $h \colon \calE \to \Rm$ are smooth ($C^\infty$). The feasible set is denoted by 
\begin{equation}\label{eq:M}
\Mcal = \{x\in \calE: h(x) = 0\}. 
\end{equation}
Our aim is to propose an infeasible algorithm for problem~\eqref{eq:P} that has good global complexity guarantees---an active topic of research whose related literature is described in Appendix~\ref{sec:literature}. 
 
The complexity is expressed in terms of worst-case number of iterations needed to find an $\varepsilon$-approximate (second-order) critical point. Thus, we need a precise notion of approximate criticality. For constrained problems such as~\eqref{eq:P}, especially when it comes to second-order criticality, there does not seem to be a consensus on what that should be: Appendix~\ref{sec:literature} reviews various proposals that have been made, with their pros and cons.

Here, under a certain LICQ-type assumption (see~\aref{assu:ROI} below), we propose a natural notion of $\varepsilon$-approximate second-order optimality conditions in Section~\ref{sec:optimality_condition}. Our definition has a geometric interpretation, as it is an extension of the Riemannian optimality conditions to points that are approximately feasible. This allows us to use the formalism of Riemannian optimization for the complexity analysis of an infeasible method. This perspective, combined with a modern take on some of Fletcher's ideas from the 1970s, leads to improved complexity bounds.
 
Concretely, we propose an algorithm which computes such $\varepsilon$-approximate second-order critical points with state-of-the-art worst-case iteration complexity with respect to $\varepsilon$ (Section~\ref{sec:algorithm}). The algorithm relies on an augmented Lagrangian formalism introduced by~\citet{fletcher1970class} which provides a \emph{smooth yet exact} penalty function for constrained optimization given by
\begin{align}
\label{eq:g_page1}
g(x) := f(x) - \inner{h(x)}{\lambda(x)} + \beta \norm{h(x)}^2,
\end{align}
for some parameter $\beta\geq 0$ and multipliers $\lambda(x)$ defined below in~\eqref{eq:lambda}. This penalty function has a reputation for being impractical. We study it in its original form as it allows us to secure desirable theoretical guarantees. Moreover, we note that other authors~\citep{gao2019parallelizable,estrin2020implementing,estrin2020implementingb} have successfully used approximations of $g$ or its derivatives to build practical schemes, and it is possible that the theoretical guarantees could extend to those as well.

Our theoretical algorithm is a simple method that combines gradient and eigensteps applied to Fletcher's augmented Lagrangian. The existing literature on Fletcher's augmented Lagrangian focuses on asymptotic convergence towards minimizers. We complement this with an analysis of the global complexity of computing approximate critical points of~\eqref{eq:P} (non-asymptotic). We do so in two phases.

First, we show that approximate critical points of Fletcher's penalty are approximate critical points of~\eqref{eq:P}. This is an extension of known results which relate exact minimizers of $g$ to exact minimizers of~\eqref{eq:P}. Second, we show that our algorithm computes approximate minimizers of $g$ in finite time, and we give a worst-case bound on the number of iterations for their computation. This leads to a  complexity rate with respect to $\varepsilon$ which improves on the state of the art for computing second-order critical points under equality constraints (even after taking into account the differences in notions of approximate criticality). One downside of the algorithm is that it requires properly setting a penalty parameter $\beta$: we discuss how to circumvent this issue in Section~\ref{sec:plateau-scheme}, at the cost of log-factors.

There is a wealth of related literature. We provide a detailed review in Appendix~\ref{sec:literature}, both regarding different notions of approximate criticality for constrained optimization and regarding iteration complexity. These complexity results are further summarized in Table~\ref{table:review}. For our contributions and outline of the paper, see Section~\ref{sec:contributions}. We preface this with our assumptions on Problem~\eqref{eq:P} in Section~\ref{sec:assumptions} and our geometric definition of approximate critical points in Section~\ref{sec:optimality_condition}.

\subsection{Assumptions}
\label{sec:assumptions}
We introduce three central assumptions about the set $\M$~\eqref{eq:M}. The following set is open:
\begin{align}
\mathcal{D} = \{ x \in \calE \colon \rank(\D h(x)) = m \}.
\label{eq:D}
\end{align}
 It is known that if $\M $ is included in $\calD$ then $\M$ is a (smooth) embedded submanifold of $\calE$~\citep{Absil2008}. We further assume that there is a region around $\M$ where the differential of the constraints is nonsingular: this is the classical linear independence constraint qualification (LICQ). Below we use $\norm{\cdot}$ to denote the  2-norm on $\R^m$. 
\begin{assumption} \label{assu:ROI}
	There exist constants $R, \sigmabar > 0$ such that for all $x$ in the set
		\begin{align}
			\calC & = \{ x \in \calE : \|h(x)\| \leq R \}
			\label{eq:C}
		\end{align}
		we have $\sigmamin(\D h(x)) = \sigma_m(\D h(x)) \geq \sigmabar > 0$ where $\sigma_k(A)$ and $\sigmamin(A)$ denote the $k$\textrm{th} and the smallest singular value of a linear map $A$, respectively. In particular, $\calM \subseteq \calC \subset \calD$.
\end{assumption}
\begin{assumption}\label{assu:boundedMandC}
The sets $\Mcal = \{x\in \calE: h(x) =0\}$ and $\calC = \{ x \in \calE : \|h(x)\| \leq R \}$ are compact.
\end{assumption}
Rather than assuming that $\M$ and $\calC$ are compact, our results could be extended to assume instead that the sublevel set $\left\{x\in \calE\colon g(x) \leq g(x_0)\right\}$ is bounded. We choose to proceed with~\aref{assu:boundedMandC} to avoid assumptions which would mix $h$ and $f$ (as is the case for $g$). In doing so, we favor assumptions that can be checked once for a given constraint $h$, leading to results that apply for broad choices of cost function $f$. 

\begin{assumption} \label{assu:lipschitzh}
	There exists a constant $C_h > 0$ such that, for all $x\in \Ccal$ and $v\in \calE$,
	\begin{align*}
	h(x + v) & = h(x) + \D h(x)[v] + E(x, v)
	\end{align*}
	with $\|E(x, v)\| \leq C_h \|v\|^2$.
\end{assumption}
Given the nonconvex nature of~\eqref{eq:P}, it is necessary to make some assumption in order to guarantee convergence to a feasible point. The set $\calC$ is the region where our assumptions apply. Accordingly, we require initialization in $\calC$. In some cases, such initializations are easy to produce (see Stiefel example below); in other cases, one may resort to a two-phase algorithm whose first phase attempts to compute an approximately feasible point~\citep{cartis2019optimality}. 
\begin{assumption}\label{assu:x0}
The iterate $x_0$ belongs to $\calC$.
\end{assumption}

Note that affine constraints do not satisfy~\aref{assu:boundedMandC}, but these constraints are usually not problematic as there are various effective approaches to handle them, including feasible methods. The following example shows how to compute the constants $R$ and $\sigmabar$, which define the region of interest $\calC$, for the Stiefel manifold. 
\begin{arxiv} 
The derivations can be found in Appendix~\ref{sec:proofs_intro}. 
\end{arxiv}
\begin{example*}[The Stiefel  manifold] \label{example:stiefel}
Let $\mathcal{E} = \R^{n\times p}$ for $1\leq p\leq n$. The Stiefel manifold is defined as 
\begin{equation}
\St(n,p) = \{X\in \R^{n\times p}: X\transpose X = \I_p\}. 
\end{equation}
The manifold corresponds to the defining function $h\colon \R^{n\times p} \to \mathrm{Sym}(p)\colon X \mapsto h(X) = X^\top X - \I_p$, where $\mathrm{Sym}(p)$ is the set of symmetric matrices of size $p$. For any $R<1$, it is possible to verify that all $X\in \R^{n\times p}$ such that $\norm{h(X)} \leq R$ satisfy $\sigmamin(\D h(X)) \geq 2 \sigmamin(X) \geq 2\sqrt{1-R}$. Therefore,~\aref{assu:ROI} is satisfied for any $R<1$ and $\sigmabar \leq 2\sqrt{1-R}$. For any $R>0$, the set $\St(n,p)$ satisfies the compactness assumption~\aref{assu:boundedMandC}. Assumption~\aref{assu:lipschitzh} holds with $C_h=1$. Additionally,~\aref{assu:x0} is easily satisfied by taking a matrix with $p$ orthonormal columns in $\Rn$ as initial iterate.\end{example*}
Given several sets that satisfy the assumptions above, their Cartesian product also does.
\begin{proposition}\label{prop:product_manifold}
For $i=1,2,\dots,k$, consider $k$ functions $h_i\colon \calE_i \to \R^{m_i}$ that satisfy assumptions~\aref{assu:ROI}, \aref{assu:boundedMandC} and \aref{assu:lipschitzh} with constants $R_i$, $\sigmabar_i$ and $C_{h_i}$. Then, the function $h\colon \calE \to \Rm$ with $\calE= \calE_1 \times \cdots \times \calE_k$ and $m=m_1 + \cdots +m_k$ defined by $h(x_1, \dots,x_k) = \left(h_1(x_1), \dots, h_k(x_k)\right)\transpose$ satisfies~\aref{assu:ROI}, \aref{assu:boundedMandC} and \aref{assu:lipschitzh} with constants $R = \min(R_1, \dots, R_k)$, $\sigmabar = \min(\sigmabar_1, \dots,\sigmabar_k)$ and $C_h = \max(C_{h_1}, \dots, C_{h_k})$.
\end{proposition}
\begin{proof}
Let $\calC = \left\lbrace x\in \calE\colon \norm{h(x)}\leq R\right\rbrace$ and $x = (x_1, \dots,x_k)\in \calC$. Since $\norm{h(x)}\leq \min (R_1,\dots,R_k)$, we have $\norm{h_i(x_i)}\leq R_i$ and $\sigmamin\!\left(\D h_i(x_i)\right) \geq \sigmabar_i$ for all $i=1,\dots,k$. This implies $\sigmamin\!\left(\D h(x)\right) \geq \min(\sigmabar_1, \dots,\sigmabar_k)$. The Cartesian product of compact sets is compact. Let $v = (v_1, \dots, v_k) \in \calE$. Since $v_i\in \calE_i$ and $x_i\in \calC_i$, we have $h_i(x_i + v_i)  = h_i(x_i) + \D h_i(x_i)[v_i] + E_i(x_i, v_i)$ with $\|E_i(x_i, v_i)\| \leq C_{h_i} \|v_i\|^2$ for all $i=1,\dots,k$. Define $E(x,v) = \left(E_1(x_1,v_1),\dots, E_k(x_k,v_k)\right)\transpose$, then $h(x + v) = h(x) + \D h(x)[v] + E(x, v)$ with $\norm{E(x,v)}\leq \sum_{i=1}^k \norm{E_i(x_i,v_i)} \leq \sum_{i=1}^k C_{h_i} \norm{v_i}^2 \leq \left(\max_{i} C_{h_i}\right) \norm{v}^2$. 
\end{proof}
As the Stiefel manifold includes spheres ($p=1$) and orthogonal matrices ($p=n$) as special cases, Proposition~\ref{prop:product_manifold} establishes that products of spheres and orthogonal/rotation groups satisfy the assumptions above. This covers a wide range of applications, of which we list a few. A product of spheres appears in the Burer-Monteiro factorization of semidefinite programs with diagonal constraints~\citep{burer2003nonlinear}. It also appears in independent component analysis (ICA) and orthogonal tensor decomposition~\citep{ge2015escaping}, and models the rank reduction of correlation matrices~\citep{grubivsic2007efficient}. Products of orthogonal matrices appear in applications of orthogonal group synchronisation~\citep{ling2023solving}. The simultaneous localization and mapping problem in robotics involves optimization over a product of Stiefel manifolds~\citep{Rosen2021Advances}.

\subsection{Optimality conditions on layered manifolds}
\label{sec:optimality_condition}

Assumption~\aref{assu:ROI} allows us to characterize any point in $\calC$ as belonging to some Riemannian submanifold of $\calE$. This manifold is defined by a level set of the function $h$, while the feasible set $\M$ is the zero-set of $h$. This observation partitions the region of interest $\calC$ into Riemannian submanifolds which we call \textit{layered manifolds}. These layered manifolds help to formulate meaningful criticality conditions for points which are nearly but not exactly feasible. 
\begin{proposition}[Layered manifolds]
Under~\aref{assu:ROI}, for any $x \in \calC$, the set $\calM_x = \{y\in \calE: h(y) = h(x)\}$ is a submanifold of $\calE$ contained in $\calC$. The tangent space and the normal space of $\calM_x$ at $y \in \calM_x$ are given respectively by:
\begin{align}
\T_y\calM_{x}  &= \ker \D h(y) &&\textrm{ and } & \mathrm{N}_y \M_x &= \spann\!\left( \D h(y)^*\right),
\end{align}
where a star indicates an adjoint.
\end{proposition}
\begin{proof}
Using Proposition 3.3.3 from~\citep{Absil2008}, the set $\M_x$ is a submanifold of $\calE$ if $\rank(\D h(y)) = m$ for all $y\in \calM_x$, which holds for all $x\in \calC$ under~\aref{assu:ROI}.
\end{proof}
The embedded submanifold $\M_x$ for some $x\in \calC$ is turned into a Riemannian submanifold using the Euclidean inner product of $\calE$ restricted to the tangent spaces of $\M_x$. We proceed to compute the Riemannian gradient and Riemannian Hessian of $f$ on the layer $\M_x$. To this end, we define the function $\lambda \colon \calE \to \Rm$ as follows:
\begin{align}
	\lambda(x) & = (\D h(x)^*)^\dagger[\nabla f(x)],
	\label{eq:lambda}
\end{align}
where a dagger indicates a Moore--Penrose pseudo-inverse. \emph{This is the same function $\lambda(\cdot)$ used in Fletcher's augmented Lagrangian} (Equation~\eqref{eq:g_page1}). This function is particularly relevant at points $x$ in $\calC$ because, if $\rank \D h(x) = m$, then the orthogonal projector from $\calE$ to the tangent space $\T_x\calM_{x} = \ker \D h(x)$ is given in explicit form by
	\begin{align*}
		\Proj_x(v) & = v - \D h(x)^*[z] & \textrm{ with } && z & = (\D h(x)^*)^\dagger[v].
	\end{align*}
Therefore, the Riemannian gradient of $f$ on $\calM_x$ is given by
	\begin{align}
		\grad_{\calM_x} f(x) & = \Proj_x\!\left( \nabla f(x) \right) = \nabla f(x) - \D h(x)^*[\lambda(x)],
		\label{eq:gradfM}
	\end{align}
	 the orthogonal projection of the Euclidean gradient of $f$ to the tangent space $\T_x \calM_x$. Likewise, the Riemannian Hessian of $f$ on $\calM_x$ is given by
	\begin{align}
		\Hess_{\calM_x} f(x) & = \Proj_x \circ \left(  \nabla^2 f(x) - \sum_{i = 1}^{m} \lambda_i(x) \nabla^2 h_i(x) \right) \circ \Proj_x,
		\label{eq:hessfM}
	\end{align}
	a self-adjoint linear operator on $\T_x\calM_x$~\cite[Section 7.7]{boumal2023intromanifolds}.

We now go over exact and approximate criticality conditions for problem~\eqref{eq:P}.
First-order critical points of~\eqref{eq:P} are defined by
\begin{align}
h(x)  &= 0 & \textrm{ and } && \grad_\calM f(x) = 0,
\label{eq:riemannian-feasible-focp}
\end{align}
whereas second-order critical points satisfy
\begin{align}
h(x) & = 0, & \grad_\calM f(x) & = 0, & \textrm{ and } & & \Hess_\calM f(x) & \succeq 0.
\label{eq:riemannian-feasible-socp}
\end{align}
At points $x\in\calD$~\eqref{eq:D}, constraint qualifications hold, providing:
\begin{proposition}
Any local minimizer of~\eqref{eq:P} is a second-order critical point which satisfies~\eqref{eq:riemannian-feasible-socp}.
\end{proposition}
Using this Riemannian viewpoint, we propose a new definition of approximate criticality for smooth equality constraints. We compare this new notion to existing ones in Appendix~\ref{sec:literature}.
\begin{definition}  \label{def:approx_fosp} The point $x\in \calD$ is an $(\varepsilon_0, \varepsilon_1)$-approximate first-order critical point of~\eqref{eq:P} if
\begin{align}
\norm{h(x)} &\leq \varepsilon_0 & \textrm{ and } && \norm{\grad_{\M_{x}} f(x)}\leq\varepsilon_1.
\label{eq:FOCP}
\tag{$\varepsilon$-FOCP}
\end{align}
\end{definition}
\begin{definition} \label{def:approx_sosp} The point $x\in \calD$ is an $(\varepsilon_0, \varepsilon_1, \varepsilon_2)$-approximate second-order critical point of~\eqref{eq:P} if
\begin{align}
\norm{h(x)} &\leq \varepsilon_0, & \norm{\grad_{\M_x} f(x)} &\leq \varepsilon_1 &  \textrm{ and } && \Hess_{\M_x} f(x) &\succeq -\varepsilon_2\Id.
\label{eq:SOCP}
\tag{$\varepsilon$-SOCP}
\end{align}
\end{definition}
The notions of~\eqref{eq:FOCP} and~\eqref{eq:SOCP} have a natural geometric interpretation. For a point $x \in \calC$ which is nearly feasible, the criticality is assessed with respect to the manifold layer to which $x$ belongs. In essence, $x$ satisfies the usual approximate criticality conditions for a Riemannian optimization problem, i.e., small Riemannian gradient and almost positive semi-definite Riemannian Hessian. However, these conditions are satisfied on the tangent space of a layer manifold $\calM_x$ rather than on the target manifold $\calM$.

\subsection{Contributions}
\label{sec:contributions}
We summarize our contributions in the following list:
\begin{itemize}
\item We propose a new definition of approximate criticality for~\eqref{eq:P}, see~\eqref{eq:FOCP} and~\eqref{eq:SOCP}. These conditions are an extension of Riemannian optimality conditions to points outside the feasible manifold $\M$. We believe that these conditions are more natural geometrically than commonly used conditions for constrained optimization problems. Approximate criticality in our sense implies approximate criticality in the more common sense (with the same $\varepsilon$), but the converse is not true (unless $\varepsilon$ is smaller than some unknown threshold)---we give an example in Appendix~\ref{sec:literature}.
\item We relate the landscape of~\eqref{eq:P} with the the landscape of $g$ (Fletcher's augmented Lagrangian, Equation~\eqref{eq:g_page1}). As far as we know, the existing literature on Fletcher's augmented Lagrangian is limited to asymptotic convergence results. Those rely on the convenient property that exact minimizers of Fletcher's augmented Lagrangian are exact minimizers of~\eqref{eq:P} under suitable conditions (see Proposition~\ref{prop:bertsekas}). In contrast, to obtain complexity bounds, it is necessary to consider approximate minimizers. In Section~\ref{sec:fletcher-alm}, we show that \emph{approximate} first- and second-order critical points of Fletcher's augmented Lagrangian satisfy~\eqref{eq:FOCP} and~\eqref{eq:SOCP} for~\eqref{eq:P}, provided that the penalty parameter $\beta$ is large enough.
\item We apply a standard unconstrained minimization algorithm to $g$—with small modifications to remain in the set $\calC$~\eqref{eq:C}—and leverage our observations from the previous points to deduce guarantees about problem~\eqref{eq:P}. Algorithm~\ref{algo:gradient-eigenstep} finds points which satisfy~\eqref{eq:FOCP} and~\eqref{eq:SOCP} for~\eqref{eq:P} in a worst-case iteration complexity which improves on the state of the art. Our main complexity result is in Theorem~\ref{thm:complexity}. Informally, it states the following:
\begin{theorem}[Informal statement]Under \aref{assu:ROI}, \aref{assu:boundedMandC}, \aref{assu:lipschitzh}, \aref{assu:x0}, given $\beta>0$ large enough, Algorithm~\ref{algo:gradient-eigenstep} produces an $(\varepsilon_1, 2\varepsilon_1)$-$\mathrm{FOCP}$ of~\eqref{eq:P} in at most  $\bigO\left(\varepsilon_1^{-2}\right)$ iterations. Algorithm~\ref{algo:gradient-eigenstep} also produces an $(\varepsilon_1, 2\varepsilon_1,\varepsilon_2 + C\varepsilon_1)$-$\mathrm{SOCP}$ of~\eqref{eq:P} in at most $\bigO\left(\max\lbrace \varepsilon_1^{-2},\varepsilon_2^{-3}\rbrace\right)$ iterations, where $C\geq 0$ is a constant depending on the constraint function $h$.
\end{theorem}
In Section~\ref{sec:plateau-scheme} we show how to estimate $\beta$ at the cost of a log-factor in complexity. 
\item We provide an extensive review of complexity guarantees for nonconvex optimization with equality constraints in Appendix~\ref{sec:literature}. The appendix details how our complexity results compare with the existing literature. Algorithm~\ref{algo:gradient-eigenstep} is the first augmented Lagrangian method to find approximate second-order critical points in a total iteration complexity of $\bigO(\varepsilon^{-3})$. Beyond augmented Lagrangian methods, the only other method that we are aware of which achieves a similar complexity is the two-phase method in~\citep{cartis2019optimality}; however, this method achieves a markedly different notion of criticality which makes the comparison delicate. Table~\ref{table:review} summarizes some of the main results detailed in Appendix~\ref{sec:literature}.
\end{itemize}
Admittedly, the class of problems to which our theoretical algorithm applies is somewhat limited. These limitations should be compared to those of other algorithm with comparable theoretical guarantees, which also feature restrictive assumptions. As described in Appendix~\ref{sec:literature}, other works face similar limitations in terms of initialization, smoothness and compactness assumptions. When other works are able to relax some of these assumptions, it typically comes with a worsening of the iteration complexity.

\begin{sidewaystable}
\begin{tabu}{|c|c|c|c|c|c|}
\hline 
 Paper & local rate & complexity & target points & problem class & 2nd order \\ 
\hline
This work & \xmark$^*$ & $\bigO(\varepsilon^{-2})$ and $\bigO(\varepsilon^{-3})$ &  \eqref{eq:SOCP}    &$\min f(x)$ s.t. $h(x)=0$  & \cmark \\ 
\hline 
\citep{polyak2009local} &  quadratic & \xmark & \xmark & $\min f(x)$ s.t. $h(x)=0$ & \cmark \\ 
\hline 
\citep{cartis2019optimality} & \xmark & $\bigO(\varepsilon^{-2})$ and $\bigO(\varepsilon^{-3})$ & \eqref{eq:cartis-critical} & $\min f(x)$ s.t. $h(x)=0, x\in C$ cvx & \cmark \\ 
\hline 
 \citep{xie2021complexity} & \xmark &  $\bigO(\varepsilon^{-7})$  & \eqref{eq:lagrange-eps-focp} and \eqref{eq:lagrange-eps-socp} & $\min f(x)$ s.t. $h(x)=0$  & \cmark \\ 
\hline 
 \citep{cifuentes2019polynomial} & \xmark  & $\bigO(\varepsilon^{-6})$  & \eqref{eq:Cifuentes-critical}  & BM for SDP & \cmark \\ 
\hline 
\citep{andreani2007second}  & \xmark  &  \xmark & \eqref{eq:SOCP} + ineq. & $\min f(x)$ s.t. $h(x)=0, h_2(x) \leq 0$ & \cmark  \\ 
\hline 
\citep{xiao2020class}  & quadratic & \xmark  & \eqref{eq:riemannian-feasible-socp} & $\min f(X)$ s.t. $X\in \St(n,p)$   & \cmark  \\ 
\hline 
 \citep{grapiglia2019complexity} & \xmark & $\bigO(\varepsilon^{-2/(\alpha-1)})$, $\alpha>1^\dagger$  & \eqref{eq:FOCP} & $\min f(x)$ s.t. $h(x)=0, h_2(x) \leq 0$ & \xmark  \\ 
 \hline 
 \citep{gao2019parallelizable} & linear &  $\bigO(\varepsilon^{-2})$ & \eqref{eq:FOCP} & $\min f(X)$ s.t. $X\in \St(n,p)$  & \xmark  \\ 
\hline   
 \citep{bai2018analysis} &  linear & $\bigO(\varepsilon^{-4})$  & \eqref{eq:FOCP} & $\min f(x)$ s.t. $h(x)=0$  & \xmark \\ 
\hline 
  \citep{bai2019proximal}& linear  & \xmark & \eqref{eq:riemannian-feasible-focp} & $\min f(h(x))$ s.t. $\Acal(h(x))=b$, $f$ cvx & \xmark \\ 
\hline
\citep{birgin2019complexity}& \xmark  & $\bigO(\log(1/\varepsilon))^\ddagger$ & \eqref{eq:FOCP} + ineq. & $\min f(x)$ s.t. $h(x)=0, h_2(x) \leq 0$& \xmark \\ 
\hline
 \end{tabu}
 \caption{Summary of related works on complexity for constrained optimization. The complexity column gives the total iteration complexity to reach first-order target points and second-order when available. The last column indicates whether second-order critical points are considered. $*)$ The algorithm that we present does not come with a guarantee of local quadratic convergence. However, it is possible to modify it to ensure local quadratic convergence, see Remark~\ref{remark:local_rate}. $\dagger)$ The bound $\bigO(\varepsilon^{-2/(\alpha-1)})$ in~\citep{grapiglia2021complexity} is an outer iteration complexity. $\ddagger)$ The bound $\bigO(\log(1/\varepsilon))$ in~\citep{birgin2019complexity} assumes that the penalty parameters $\beta_k$ remain bounded as $k\to \infty$.}
  \label{table:review}
\end{sidewaystable}  

\section{Properties of Fletcher's augmented Lagrangian}
\label{sec:fletcher-alm}

We now cover properties of the function $g$, Fletcher's augmented Lagrangian (Equation~\ref{eq:g_page1}). In this section, we recall an original result from~\citep{bertsekas1982constrained} which establishes conditions under which the critical points and minimizers of $g$ and~\eqref{eq:P} are equivalent. The core of this section then establishes extensions of this result to the case of approximate critical points. That is, we show that approximate first- and second-order critical points of $g$ are also approximately critical for~\eqref{eq:P} in the sense of~\eqref{eq:FOCP} and~\eqref{eq:SOCP}. 
\paragraph{}
For problem~\eqref{eq:P}, the Lagrangian $\calL(x,\lambda)\colon \calE \times \Rm \to \R$ is defined as 
\begin{align*}
\calL(x,\lambda) &= f(x) - \inner{\lambda}{h(x)},
\end{align*}
where $\lambda\in \Rm$ is called the vector of multipliers.
The augmented Lagrangian $\calL_\beta \colon \calE \times \Rm \to \R$ for some penalty parameter $\beta \geq 0$ is:
\begin{align}
	\calL_\beta(x, \lambda) & = f(x) - \inner{\lambda}{h(x)} + \beta \|h(x)\|^2.
\end{align}
This penalty function has given rise to a number of popular methods for constrained optimization~\citep{bertsekas1982constrained}. \citet{fletcher1970class} proposed a variant, which we denote by $g$ (already shown in Equation~\eqref{eq:g_page1}):
\begin{align}
	g(x) & = \calL_\beta(x, \lambda(x)),
	\label{eq:g}
\end{align}
where $\lambda(x)$ is defined in~\eqref{eq:lambda}.
We note that the set $\Dcal$ (Equation~\eqref{eq:D}) is open, and it is easy to verify that $\lambda(\cdot)$ is $C^\infty$ on that set. We also note that, under~\aref{assu:ROI}, the set $\calC$ is included in $\calD$. Therefore, $g$ is also smooth on $\calC$. Fletcher's augmented Lagrangian is a smooth penalty which depends only on $x$, the primal variable. The multipliers are computed as a function of $x$. We define $C_\lambda(x)$ as the operator norm of the differential of $\lambda(\cdot)$. Since $\Ccal$ is assumed compact and $\lambda(\cdot)$ is smooth, this quantity is bounded.
\begin{definition}\label{def:Clambda}
Under~\aref{assu:ROI}, for any $x\in \Ccal$, we define the quantity
\begin{equation*}
 C_\lambda(x) :=\opnorm{\D\lambda(x)}= \sigma_1\left(\D\lambda(x)\right).
\end{equation*}
Additionally, under~\aref{assu:boundedMandC}, we define the constant
\begin{equation*}
 \Clambdabar := \max_{x\in \Ccal} \opnorm{\D\lambda(x)}<\infty.
\end{equation*}
\end{definition}

\begin{definition}
Under~\aref{assu:ROI},	for $x\in \calC$, we define the following quantities
\begin{align}
\beta_1(x) &=  \dfrac{\sigma_1(\D h(x))C_\lambda(x)}{2\sigmamin^2(\D h(x))} \\
\beta_2(x) &=  \dfrac{C_\lambda(x)}{\sigmamin(\D h(x))} \\
\beta_3(x) &= \dfrac{1}{\sigmamin(\D h(x))}.
\end{align}
Additionally, under~\aref{assu:boundedMandC}, we define the  constants $\bar{\beta_i} = \max_{x\in \calC} \beta_i(x)$ for $i = 1,2,3$.
\label{def:beta123}
\end{definition}

The following classical result connects first-order critical points and minimizers of $g$ and~\eqref{eq:P}.
\begin{arxiv}
A proof of this original result using our notations can be found in Appendix~\ref{sec:appendix-bertsekas}.
\end{arxiv}

\begin{proposition}[\citep{bertsekas1982constrained},~Prop.~4.22] 
\label{prop:bertsekas}
	Let $g(x) = \Lcal_\beta(x, \lambda(x))$ be Fletcher's augmented Lagrangian and assume $\M \subset \Dcal$, where $\Dcal = \{ x \in \Ecal : \rank( \D h(x)) = m \}$ and $\M = \left\{ x\in \calE: h(x) = 0\right\}$. 
	\begin{enumerate}
		\item For any $\beta$, if $x$ is a first-order critical point of~\eqref{eq:P}, then $x$ is a first-order critical point of $g$.
		\item Let $x \in \Dcal$ and $\beta > \beta_1(x)$. If $x$ is a first-order critical point of $g$, then  $x$ is a first-order critical point of~\eqref{eq:P}.
		\item Let $x$ be a first-order critical point of~\eqref{eq:P} and let $K$ be a compact set. Assume $x$ is the unique global minimum of $f$ over $\M \cap K$ and that $x$ is in the interior of $K$. Then, there exists $\beta$ large enough such that $x$ is the unique global minimum of $g$ over $K$.
		\item Let $x\in \calD$ and $\beta > \beta_1(x)$. If $x$ is a local minimum of $g$, then $x$ is a local minimum of~\eqref{eq:P}.
	\end{enumerate}
\end{proposition}
The previous shows that minimizing the function $g$ inside $\calD$ provides a way to find minimizers of~\eqref{eq:P}. However, in practice, algorithms can only find approximate first- and second-order critical points in finite time. With the above proposition, one is left wondering whether such approximate points for $g$ correspond to similarly approximate critical points for~\eqref{eq:P}. The remainder of this section provides such guarantees.

\subsection{Approximate first-order criticality}
In this section, we show that if $\nabla g(x)$ is small at some $x\in \calC$, the point $x$ is approximately first-order critical for~\eqref{eq:P} in the sense of~\eqref{eq:FOCP}. We begin with a straightforward computation of the gradient of $g$. The gradient of the augmented Lagrangian $\calL_\beta$ with respect to $x$ is given by
\begin{align}
	\nabla_x \calL_\beta(x, \lambda) & = \nabla f(x) - \D h(x)^*[\lambda] + 2\beta \D h(x)^*[h(x)] \nonumber\\
									 & = \nabla f(x) - \D h(x)^*[\lambda - 2\beta h(x)].
\end{align}
Owing to~\eqref{eq:gradfM}, we make the following central observation: the gradient of $\calL_\beta$ with respect to its first argument, when evaluated at $(x, \lambda(x))$, splits into orthogonal components; one component in the tangent space $\T_x\calM_x$, and one component in the normal space to $\calM_x$ at $x$:
\begin{align}
	\nabla_x \calL_\beta(x, \lambda(x)) & = \grad_{\calM_x} f(x) + 2\beta  \D h(x)^*[h(x)].
	\label{eq:nablalagrangesplit}
\end{align}
Owing to orthogonality, $\nabla_x \calL_\beta(x, \lambda(x))$ is small if and only if the two terms on the right are small. It takes an easy computation to check that for all $x \in \calD$ we have
\begin{align*}
	\D g(x)[v] & = \D f(x)[v] - \inner{\D \lambda(x)[v]}{h(x)} - \inner{\lambda(x)}{\D h(x)[v]} + 2 \beta \inner{h(x)}{\D h(x)[v]} \\
			& = \inner{\nabla f(x)}{v} - \inner{\D h(x)^*\!\left[\lambda(x) - 2\beta h(x)\right]}{v} - \inner{\D \lambda(x)^*\!\left[h(x)\right]}{v} \\
			& = \inner{\nabla_x \calL_\beta(x, \lambda(x))}{v} - \inner{\D \lambda(x)^*\!\left[h(x)\right]}{v}.
\end{align*}
Thus, for all $x \in \calD$,
\begin{align}
	\nabla g(x) & = \nabla_x \calL_\beta(x, \lambda(x)) - \D \lambda(x)^*\!\left[h(x)\right] \nonumber\\
					 & = \grad_{\calM_x} f(x) + 2\beta \D h(x)^*[h(x)] - \D \lambda(x)^*\!\left[h(x)\right].
					 \label{eq:gradg}
\end{align}
Therefore, for $x\in \calM$, $\nabla g(x) = \grad_\calM f(x)$. Consequently, for any value of $\beta$, if $x$ satisfies the constraints $h(x) = 0$, that is, if $x$ is on the manifold $\calM$, then $x$ is first-order critical for $f$ on $\calM$ (Equation~\eqref{eq:riemannian-feasible-focp}) if and only if $\nabla g(x)=0$.
We now add to this with a claim about approximate first-order critical points of $g$.
\begin{proposition}\label{thm:g-approx-stationary}
Under~\aref{assu:ROI}, take $\varepsilon_1\geq 0$ and $x\in \Ccal$ with $\beta > \max\!\left \lbrace \beta_2(x), \beta_3(x)\right \rbrace$. 
 If $\norm{\nabla g(x)}\leq \varepsilon_1$, then $x$ is an $(\varepsilon_1, 2\varepsilon_1)-$approximate first-order critical point of~\eqref{eq:P} (see~\eqref{eq:FOCP}) as
 \begin{align}
  \norm{h(x)} &\leq \dfrac{\varepsilon_1}{\beta \sigmamin(\D h(x))}  \leq  \varepsilon_1 & \textrm{ and } && \norm{\grad_{\Mcal_x} f(x)} &\leq \left(1+\dfrac{C_\lambda(x)}{\beta \sigmamin(\D h(x))}\right)\varepsilon_1\leq  2\varepsilon_1.
\end{align}
\end{proposition}
\begin{proof}
We remember from~\eqref{eq:gradg}	that
\begin{align}
\nabla g(x) &= \grad_{\Mcal_x} f(x) + 2\beta \D h(x)^*[h(x)] - \D\lambda(x)^*[h(x)]\\
&=  \grad_{\Mcal_x} f(x) - \Proj_x \left( \D\lambda(x)^*[h(x)] \right) + 2\beta \D h(x)^*[h(x)] - \Proj_x^\perp \left( \D\lambda(x)^*[h(x)]\right)
\end{align}
where $\Proj_x^\perp = \Id - \Proj_x$, is the orthogonal projection on $\Nrm_x \M_x = \left(\T_x \M_x\right)^\perp$, the normal space to $\M_x$ at $x$. We have decomposed the right-hand side in two tangent and two normal terms with respect to the manifold $\M_x$. By orthogonality, $\norm{\nabla g(x)}\leq \varepsilon_1$ implies that both the tangent and normal components have norm smaller than $\varepsilon_1$. For the normal terms this yields,
\begin{align}
\norm{ 2\beta \D h(x)^*[h(x)] - \Proj_x^\perp \left(\D\lambda(x)^*[h(x)]\right)} &=
\norm{ \left(2\beta \D h(x)^* - \Proj_x^\perp \left(\D\lambda(x)^*\right) \right)[h(x)]} \leq \varepsilon_1.
\label{eq:stationary_g_normal}
\end{align}
Note that $\D h(x)^*$ is nonsingular since $x\in\Ccal$. We show that $\beta$ is large enough so that the operator  $\left(2\beta \D h(x)^* - \Proj_x^\perp \left(\D\lambda(x)^*\right) \right)$ is nonsingular. We use Weyl's inequality to control singular values (\citet[Theorem 3.3.16 (c)]{horntopics}), which states that for two linear operators  $A,B\colon \calE_1 \to \calE_2$, it holds that $\sigma_{q}(A-B) \geq \sigma_{q}(A) - \sigma_{1}(B)$ with $q \leq \min(m,n)$. This allows to write
\begin{equation}\label{eq:weyl}
\sigmamin\!\left(2\beta \D h(x)^* - \Proj_x^\perp \left(\D\lambda(x)^*\right) \right)\! \geq \sigmamin\!\left( 2\beta \D h(x)^* \right)\!- \sigmamax \!\left( \Proj_x^\perp \left(\D\lambda(x)^*\right)\right)\!.
\end{equation}
The assumption on $\beta$ then provides
\begin{align}
 \sigmamin\!\left( 2\beta \D h(x)^* \right)\!-\sigmamax\!\left( \Proj_x^\perp \left(\D\lambda(x)^*\right)\right)\! &\geq 2\beta  \sigmamin\!\left( \D h(x) \right)\!- C_{\lambda}(x)\\
 &> \beta  \sigmamin\!\left( \D h(x) \right)\!>1.
\end{align}
We inject this into~\eqref{eq:stationary_g_normal} to find:
\begin{align}
\norm{h(x)} &\leq \dfrac{\varepsilon_1}{\sigmamin\left(2\beta \D h(x)^* - \Proj_x^\perp \left(\D\lambda(x)^*\right) \right)}\\
&\leq \dfrac{\varepsilon_1}{\beta \sigmamin(\D h(x))} \leq \varepsilon_1.
\end{align}
Now we use the tangent terms:
\begin{align}
\varepsilon_1 &\geq \norm{  \grad_{\Mcal_x} f(x) - \Proj_x \left( \D\lambda(x)^*[h(x)]\right) }\nonumber \\
 &\geq  \norm{  \grad_{\Mcal_x} f(x)}  - \norm{\Proj_x \left( \D\lambda(x)^*[h(x)]\right) }\nonumber \\
  &\geq  \norm{  \grad_{\Mcal_x} f(x)}  - \norm{  \D\lambda(x)^*[h(x)] }\nonumber \\
    &\geq  \norm{  \grad_{\Mcal_x} f(x)}  - C_\lambda(x) \norm{ h(x) } \nonumber\\
 &\geq  \norm{  \grad_{\Mcal_x} f(x)}   -  C_\lambda(x)  \dfrac{\varepsilon_1}{\beta \sigmamin(\D h(x))}.
\end{align}
This allows to conclude $  \norm{  \grad_{\Mcal_x} f(x)}  \leq \varepsilon_1 + \dfrac{ C_\lambda(x) }{\beta \sigmamin(\D h(x))} \varepsilon_1  \leq 2\varepsilon_1$.
\end{proof}

\begin{corollary}
Under~\aref{assu:ROI} and~\aref{assu:boundedMandC}, take $\varepsilon_1 \geq 0$. Let $\beta$ satisfies the global bounds
\begin{align}
\beta &>\bar{\beta_2} & \textrm{ and } && \beta &>\bar{\beta_3},
\label{eq:bound_beta_global}
\end{align}
where $\bar{\beta_2}, \bar{\beta_3}$ are introduced in Definition~\ref{def:beta123}. Any $x\in \calC$ such that $\norm{\nabla g(x)}\leq \varepsilon_1$ is an $(\varepsilon_1, 2\varepsilon_1)-$approximate first-order critical point of~\eqref{eq:P} (see~\eqref{eq:FOCP}) as
\begin{align}
  \norm{h(x)} &\leq \dfrac{\varepsilon_1}{\beta \sigmabar}  \leq  \varepsilon_1 & \textrm{ and } && \norm{\grad_{\Mcal_x} f(x)} &\leq \left(1+\dfrac{\Clambdabar}{\beta \sigmabar}\right)\varepsilon_1 \leq  2\varepsilon_1,
\end{align}
with $\sigmabar \leq \min_{x\in\calC} \sigmamin \left(\D h(x)\right)$ defined in~\aref{assu:ROI}.
\label{corollary:g-approx-stationary}
\end{corollary}

\subsection{Approximate second-order criticality}
We now turn our attention to approximate second-order critical points of Fletcher's augmented Lagrangian. Similarly to first-order criticality, we investigate connections with~\eqref{eq:SOCP} points for~\eqref{eq:P}. Specifically, we extend the observation that \emph{strict} second-order critical points of~\eqref{eq:P} and $g$ match, provided that $\beta$ is large enough. The non-strict case is less clear: see below. 

The Hessian of $g$ is obtained by taking a directional derivative of~\eqref{eq:gradg}. For any $\dot x \in \calE$,
\begin{align}
	\nabla^2 g(x)[\dot x] & = \nabla^2 f(x)[\dot x]\nonumber \\
							& \quad - \left(\D\!\left( x \mapsto \D \lambda(x)^* \right)\!(x)[\dot x]\right)[h(x)] \nonumber \\
							& \quad - \D \lambda(x)^*[\D h(x)[\dot x]] \nonumber \\
							& \quad - \left(\D\!\left( x \mapsto \D h(x)^* \right)\!(x)[\dot x]\right)[ \lambda(x) - 2\beta h(x)]\nonumber  \\
							& \quad - \D h(x)^*\!\left[ \D\lambda(x)[\dot x] - 2\beta \D h(x)[\dot x] \right].\label{eq:hessian-g}
\end{align}
We begin with a statement about feasible points which connects the Hessian of $g$ and the Riemannian Hessian of $f$ on $\calM$.
\begin{proposition}\label{prop:hessian-g-riemann-feasible}
For all $x\in \M$, with $\Proj_x$ the orthogonal projector from $\calE$ to $\T_x\calM$, we have
\begin{align}\label{eq:rhess_hessg}
\Hess_{\M} f(x) = \left.\left( \Proj_x \circ	\nabla^2 g(x) \circ \Proj_x \right)\!\right\vert_{\T_x\calM}.
\end{align}
Therefore, if $\nabla^2 g(x)\succeq -\varepsilon_2 \Id$, then $ \Hess_{\M}f(x) \succeq -  \varepsilon_2 \Id_{\T_x \calM}$. If $\nabla^2 g(x) \succ 0$, then $\Hess_\M f(x) \succ 0$. 
\end{proposition}
\begin{proof}
We show that if $h(x)=0$, then~\eqref{eq:rhess_hessg} holds. Take $\dot x \in \calE$ and plug $h(x) = 0$ into Equation~\eqref{eq:hessian-g}. This gives
\begin{align}
	\nabla^2 g(x)[\dot x] & = \nabla^2 f(x)[\dot x] - \sum_{i = 1}^{m} \lambda_i(x) \nabla^2 h_i(x)[\dot x] \nonumber \\  
							& \quad + 2\beta \D h(x)^*\!\left[ \D h(x)[\dot x] \right] \nonumber \\
							& \quad - \D \lambda(x)^*[\D h(x)[\dot x]] - \D h(x)^*\!\left[ \D\lambda(x)[\dot x] \right].
							\label{eq:hessian_g_feasible}
\end{align}
If, in addition, $\dot x \in \ker \D h(x)$, then
\begin{align*}
	\inner{\dot x}{\nabla^2 g(x)[\dot x]} & = \inner{\dot x}{\nabla^2 f(x)[\dot x]} - \sum_{i = 1}^{m} \lambda_i(x) \inner{\dot x}{\nabla^2 h_i(x)[\dot x]}.
\end{align*}
Since $\ker \D h(x) = \mathrm{T}_x \M$, we conclude from Equation~\eqref{eq:hessfM} that, restricted to $\T_x\calM$,
\begin{align*}
	\Proj_x \circ \nabla^2 g(x) \circ \Proj_x & = \Proj_x \circ \left( \nabla^2 f(x) - \sum_{i = 1}^{m} \lambda_i(x) \nabla^2 h_i(x) \right) \circ \Proj_x = \Hess_\calM f(x). \qedhere
\end{align*}\end{proof}
In particular, for $\varepsilon_2=0$, the above result tells us that, irrespective of $\beta\geq 0$, if $x\in \calM$ satisfies $\nabla g(x) =0$ and $\nabla^2 g(x) \succeq 0$, the point $x$ is second-order critical for $f$ on $\calM$ (Equation~\eqref{eq:riemannian-feasible-socp}). To our knowledge, there is no evidence that the converse is true, namely, we do not know whether at a point $x\in\calM$ that satisfies~\eqref{eq:riemannian-feasible-socp} there exists a $\beta$ large enough such that $\nabla^2 g(x) \succeq 0$. \citet{fletcher1970class} showed that the converse holds for positive definite Hessians.
\begin{arxiv}
A proof can be found in Appendix~\ref{sec:appendix-fletcher}.
\end{arxiv}
\begin{proposition}[\citet{fletcher1970class}]\label{prop:strict-2nd-order-point}
If $x\in \M$ is a local minimizer of~\eqref{eq:P} with $\Hess_\M f(x) \succ 0$, there exists $\beta$ large enough such that $\nabla^2 g(x) \succ 0$. 
\end{proposition}
\begin{remark}[Local quadratic convergence]\label{remark:local_rate} Assume that~\aref{assu:ROI},~\aref{assu:boundedMandC} and~\aref{assu:x0} hold. For $\beta$ large enough, it is possible to apply Newton's method to Fletcher's augmented Lagrangian to achieve a local quadratic convergence rate towards isolated minimizers of~\eqref{eq:P}. Let $x^*\in \calM$ be a strict second-order critical point for~\eqref{eq:P} satisfying $\grad_\Mcal f(x^*)  = 0$ and $\Hess_\calM f(x^*) \succ 0$. Provided $\beta$ is large enough, $x^*$ satisfies $\nabla g(x^*)=0$ and $\nabla^2 g(x^*) \succ 0$ (Propositions~\ref{prop:bertsekas} and~\ref{prop:strict-2nd-order-point}). 
Take $x_0\in \calC$ close enough to $x^*$, the classical Newton method applied to the function $g$ produces a sequence which converges towards $x^*$ at a quadratic rate, as is discussed  in~\citep{estrin2020implementing}.
\end{remark}
Proposition~\ref{prop:hessian-g-riemann-feasible} can be generalized to infeasible points in $\calC$ using an upper bound on the gradient norm of $g$.
\begin{proposition}\label{thm:approx_stationarity_hessian}
Under~\aref{assu:ROI},~take $x\in \Ccal $ with $\beta> \max \left \lbrace \beta_2(x),\beta_3(x)\right\rbrace$. Assume $\norm{\nabla g(x)}\leq \varepsilon_1$  so that Proposition~\ref{thm:g-approx-stationary} applies at $x$. If $\nabla^2 g(x) \succeq - \varepsilon_2 \Id$, then $x$ is an $(\varepsilon_1, 2\varepsilon_1,\varepsilon_2 + C(x)\varepsilon_1)-$approximate second-order critical point of~\eqref{eq:P} (see~\eqref{eq:SOCP}) as
\begin{equation}
\Hess_{\M_x}f(x) \succeq  -(\varepsilon_2  + C(x)  \varepsilon_1) \Id,
\end{equation}
where $C(x) =2\opnorm{\left(\D\!\left( x \mapsto \D h(x)^* \right)\!(x)\right)}/ \sigmamin(\D  h(x)) + \opnorm{\left(\D\!\left( x \mapsto \D \lambda(x)^* \right)\!(x)\right)}.$
\end{proposition}
\begin{proof}
Since $x\in \calC$, for any $\dot x \in \T_x\Mcal_x$, Equation~\eqref{eq:hessfM} gives the Riemannian Hessian of $f$ at $x$ and yields
\begin{equation}
\langle \dot x, \Hess_{\Mcal_x} f(x)[\dot x]\rangle  = \left\langle \dot x , \nabla^2 f(x)[\dot x] -  \sum_{i=1}^m \lambda_i(x) \nabla^2 h_i(x)[\dot x]\right\rangle.
\end{equation}
By assumption, for any $\dot x \in \calE$,
\begin{equation}
\langle \dot x, \nabla^2 g(x)[\dot x]\rangle \geq - \varepsilon_2 \norm{\dot x}^2.
\end{equation}
In Equation~\eqref{eq:hessian-g}, take $\dot x \in \T_x \M_x = \ker (\D h(x))$ and remember that the span of $\D h(x)^*$ is orthogonal to $\T_x \M_x$. This gives
\begin{align}
\nonumber
\left(\Proj_x \circ \nabla^2 g(x) \circ \Proj_x \right)\! [\dot x] &= \Proj_x \circ \left( \nabla^2 f(x)[\dot x]\right. \\
& \quad - \left(\D\!\left( x \mapsto \D h(x)^* \right)\!(x)[\dot x]\right)[ \lambda(x) - 2\beta h(x)] \\ \nonumber
& \quad - \left(\D\!\left( x \mapsto \D \lambda(x)^* \right)\!(x)[\dot x]\right)[h(x)] \Big).
\end{align}
For clarity, we write $F_h(x) =  \D\!\left( x \mapsto \D h(x)^* \right)\!(x)$ and $F_\lambda(x) =  \D\!\left( x \mapsto \D \lambda(x)^* \right)\!(x)$. We compute the derivative
\begin{equation}
 - \left(F_h(x)[\dot x]\right)[ \lambda(x) - 2\beta h(x)] = - \sum_{i=1}^m (\lambda_i(x)- 2\beta h_i(x)) \nabla^2 h_i(x)[\dot x],
\end{equation}
which gives
\begin{align}
\nonumber
\inner{\dot x}{ \Proj_x \circ \nabla^2 g(x) \circ \Proj_x [\dot x]} &= \inner{\dot x}{ \nabla^2 f(x)[\dot x]  - \sum_{i=1}^m \lambda_i(x) \nabla^2 h_i(x)[\dot x] }\nonumber\\
& \quad + \langle \dot x, 2\beta \left(F_h(x)[\dot x]\right)[ h(x)]\rangle\nonumber\\
& \quad - \langle \dot x,  \left(F_\lambda(x)[\dot x]\right)[h(x)]\rangle\nonumber \\ 
&\geq -\varepsilon_2 \norm{\dot x}^2.
\end{align}
The formula for $\Hess_{\M_x} f(x)$ has appeared on the right-hand side. Using $ \norm{h(x)} \leq \dfrac{\varepsilon_1}{\beta \sigmamin(\D h(x))}  \leq  \varepsilon_1$ from Proposition~\ref{thm:g-approx-stationary}, we conclude with
\begin{align}
  \inner{ \dot x}{  \nabla^2 f(x)[\dot x]  - \sum_{i=1}^m \lambda_i(x) \nabla^2 h_i(x)[\dot x] } &\geq -2\beta \left\langle \dot x,  \left(F_h(x)[\dot x]\right)[ h(x)]\right \rangle \nonumber \\ 
 &\quad + \left\langle \dot x, \left(F_\lambda(x)[\dot x]\right)[h(x)] \right\rangle -\varepsilon_2 \norm{\dot x}^2\nonumber \\
 &\geq -2\beta \opnorm{F_h(x)}  \norm{\dot x}^2\norm{ h(x)}\nonumber\\ 
 &\quad - \opnorm{ F_\lambda(x)} \norm{\dot x}^2 \norm{h(x)}  -\varepsilon_2 \norm{\dot x}^2\nonumber \\
  &\geq -2\beta \opnorm{F_h(x)} \norm{\dot x}^2 \dfrac{\varepsilon_1}{\beta \sigmamin(\D h(x))}\nonumber \\ 
 &\quad - \opnorm{ F_\lambda(x)} \norm{\dot x}^2 \varepsilon_1  -\varepsilon_2 \norm{\dot x}^2\nonumber \\  
 & \geq -\varepsilon_2 \norm{\dot x}^2 - \left(2  \opnorm{F_h(x)}/\sigmamin(\D h(x))\right. \\
 &\quad  \left. +  \opnorm{F_\lambda(x)} \right) \varepsilon_1 \norm{\dot x}^2.\qedhere
\end{align}
\end{proof}
\begin{corollary}
Under~\aref{assu:ROI} and~\aref{assu:boundedMandC}, let $\beta>\max\left\lbrace \bar{\beta_2},\bar{\beta_3} \right\rbrace$. Take $x\in \Ccal $ with $\norm{\nabla g(x)}\leq \varepsilon_1$ so that Corollary~\ref{corollary:g-approx-stationary} applies. If $\nabla^2 g(x) \succeq - \varepsilon_2 Id$, then $x$ is an $(\varepsilon_1, 2\varepsilon_1,\varepsilon_2 + C\varepsilon_1)-$approximate second-order critical point of~\eqref{eq:P} (see~\eqref{eq:SOCP}) as
\begin{equation}
\Hess_{\M_x}f(x) \succeq  -(\varepsilon_2  + C  \varepsilon_1) \Id,
\end{equation}
where $C = \max_{x\in \calC} 2\opnorm{\left(\D\!\left( x \mapsto \D h(x)^* \right)\!(x)\right)}/ \sigmabar + \opnorm{\left(\D\!\left( x \mapsto \D \lambda(x)^* \right)\!(x)\right)}$. 
\label{corollary:approx_stationarity_hessian-global}
\end{corollary}

\subsection{Property of the region $\calC$}
The algorithms we design and analyse in later sections are initialized in some connected component of $\calC = \{ x \in \calE : \|h(x)\| \leq R \}$, with $R$ as in~\aref{assu:ROI}, and produce iterates which remain in this same connected component. Since $\calC$ may in general have more than one such component, and since we hope in particular that our iterates converge to a feasible point, that is, to a point in $\calM = \{ x \in \calE : h(x) = 0 \}$, it is natural to wonder whether each connected component of $\calC$ intersects with $\calM$. That is indeed the case. We prove the following result in Appendix~\ref{sec:proof_feasibleC}.
\begin{proposition}\label{prop:feasibleC}
Under \aref{assu:ROI}, every connected component of $\calC$ contains a point $\bar{z}\in \calE$ such that $h(\bar{z}) = 0$.
\end{proposition}
The proof relies on the escape lemma from differential equations~\citep[Theorem A.42]{lee2018introduction} with a classical Polyak--{\L}ojasiewicz (P\L) condition along gradient flows. The latter part is similar in nature to Theorem 9 in~\citep{polyak1963gradient}, but the assumptions of the latter are too strong for our purpose. The cited theorem requires (among other things) that the P{\L} condition hold in a ball centered around $x_0$ whose size may be large depending on problem constants, while in our case, P{\L} only holds in $\calC$.

\section{Gradient-Eigenstep algorithm}
\label{sec:algorithm}
In light of the properties of the approximate minimizers of Fletcher's augmented Lagrangian established in the previous section, it would be natural to use an off-the-shelf algorithm for unconstrained minimization on the function $g$. However, we need to ensure that the iterates remain in the set $\calC$, which is not automatic. To this end, we present in this section an optimization algorithm to minimize $g$ that is designed to remain in $\Ccal$, the region of interest where $\lambda(x)$ is well defined. The algorithm alternates between gradient steps (first-order) and eigensteps (second-order) to reach approximate second-order critical points of $g$, as described in~\citep[Section 3.6]{wright2022optimization}. If the gradient norm of $g$ is large, a gradient step on $g$ is used. If the gradient of $g$ is below a tolerance, the algorithm follows a direction of negative curvature of the Hessian of $g$. Gradient steps and eigensteps must fulfil two purposes: they must guarantee a sufficient decrease of the penalty $g$ and also ensure that the next iterate remains inside $\calC$. This is detailed in Algorithm~\ref{algo:gradient-eigenstep}. Given values $\varepsilon_1 >0, \varepsilon_2 >0$, the algorithm returns a point which satisfies $\norm{\nabla g(x)} \leq \varepsilon_1$ and $\lambdamin \left( \nabla^2 g(x)\right) \geq -\varepsilon_2$. This ensures that $x$ is an $(\varepsilon_1, \varepsilon_2 + C(x)\varepsilon_1)$-SOCP of~\eqref{eq:P} according to Proposition~\ref{thm:approx_stationarity_hessian}. 

Whenever $\norm{\nabla g(x)}>\varepsilon_1$, a gradient step is used and we require that the step-length $\alpha$ satisfies a classical Armijo sufficient decrease condition:
\begin{equation}
g(x) - g(x - \alpha \nabla g(x)) \geq  c_1 \alpha \norm{\nabla g(x)}^2,
\label{eq:decrease1st}
\end{equation}
for some $0<c_1<1$. 
The backtracking procedure for gradient steps is presented in Algorithm~\ref{algo:1st-backtracking}. This is a classical backtracking modified to additionally ensure that the iterates stay in $\Ccal$, which is always possible for small enough steps, as we show in Proposition~\ref{prop:t1}.

Given $x\in \Ccal$ with $\norm{\nabla g(x)}\leq \varepsilon_1$ and $\lambdamin\left(\nabla^2 g(x)\right) < -\varepsilon_2$, a second-order step must be applied. We compute a unit-norm vector $d\in \calE$ such that $\langle d , \nabla^2 g(x)[d]\rangle < - \varepsilon_2 \norm{d}^2$. To ensure sufficient decrease, we wish to find $\alpha>0$ such that
\begin{equation}
g(x) - g(x + \alpha d) \geq  - c_2 \alpha^2 \langle d, \nabla^2 g(x)[d]\rangle,
\label{eq:decrease2nd}
\end{equation}
for some $0<c_2<1/2$. In Algorithm~\ref{algo:2nd-backtracking},	we detail the backtracking used for second-order steps. It is designed to ensure that~\eqref{eq:decrease2nd} holds and additionally that the steps are small enough for the iterates to remain in $\calC$, which is possible as we show in Proposition~\ref{prop:t2}.

We define some bounds on the derivatives of $g$, which are finite due to the smoothness of $g$ in $\calC$ and boundedness of $\calC$ (\aref{assu:boundedMandC}). 
\begin{definition}\label{def:Lg_Mg}
Under~\aref{assu:ROI} and~\aref{assu:boundedMandC}, define the constants
\begin{align}
\label{eq:Lg_M_g}
L_g &= \underset{x\in \calC}{\max} \opnorm{\nabla^2 g(x)} && \textrm{ and } & M_g =  \underset{x\in \calC}{\max} \opnorm{\nabla^3 g(x)}.
\end{align}
\end{definition}

\subsection{Algorithm}
We define Algorithm~\ref{algo:gradient-eigenstep}, a procedure which combines first- and second-order steps to minimize $g$ up to approximate second-order criticality if $\varepsilon_2 < \infty$. Setting $\varepsilon_2 = \infty$ gives a first-order version of the algorithm. To run Algorithm~\ref{algo:gradient-eigenstep}, we assume that the value of the penalty parameter $\beta$ does not change and is large enough in the following sense.
\begin{assumption}
Under~\aref{assu:ROI} and~\aref{assu:boundedMandC}, $\beta$ is chosen such that $\beta >  \betabar$ with
\begin{equation}
\betabar := \max \left\lbrace \betabar_1, \betabar_2, \betabar_3 \right\rbrace,
\label{eq:beta-critical}
\end{equation}
where $\betabar_i$ for $i=1,2,3$ are defined in Definition~\ref{def:beta123}.
\label{assu:beta_large_enough}
\end{assumption}
In Section~\ref{sec:plateau-scheme}, we show how this assumption can be removed, using an adaptive scheme for $\beta$.  

\begin{algorithm}
\caption{Gradient-Eigenstep} \label{algo:gradient-eigenstep}
\begin{algorithmic}[1]
\State \textbf{Given:} Functions $f$ and $h$, $x_0\in \Ccal$, $\beta >0$, $0 \leq \varepsilon_1 \leq R/2$ and $\varepsilon_2 \geq 0$.
\State Set $k \leftarrow 0$
\While{no optional stopping criterion triggers}
\If{$ \|\nabla  g(x_{k}) \| > \varepsilon_1$}
\State $x_{k+1} = x_k - t \nabla g(x_k)$ with $t$ given by Algorithm~\ref{algo:1st-backtracking}
\ElsIf{$\varepsilon_2 <\infty$}
\If{ $\lambdamin(\nabla^2 g(x_k)) < - \varepsilon_2$}
\State Find $d\in \calE$ such that  $\langle d, \nabla^2 g(x_k)[d]\rangle < -\varepsilon_2 \norm{d}^2$, $\langle d, \nabla g(x_k) \rangle \leq 0$ and $\norm{d} = 1$.
\State $x_{k+1} = x_k + t d$ where $t$ is given by Algorithm~\ref{algo:2nd-backtracking}.
\Else
\State \Return $x_k$ \Comment{$\norm{\nabla g(x_k)}\leq \varepsilon_1$ and $\nabla^2 g(x_k) \succeq - \varepsilon_2 \Id$}
\EndIf
\Else
\State \Return $x_k$  \Comment{$\norm{\nabla g(x_k)}\leq \varepsilon_1$}
\EndIf
\State $k \leftarrow k +1$
\EndWhile
\end{algorithmic}
\end{algorithm}

\begin{algorithm}
\caption{Gradient step backtracking, modified to stay in $\Ccal$\label{algo:1st-backtracking}}
\begin{algorithmic}[1]
\State \textbf{Given:} $x\in \calC$, $\alpha_{01}>0$, $0<c_1<1$, $0<\tau_1<1$.
\State Set $\alpha \leftarrow \alpha_{01}$
\While{true}
\If{$g(x) - g(x - \alpha \nabla g(x)) \geq  c_1 \alpha \norm{\nabla g(x)}^2$ \textbf{and} $x - \alpha \nabla g(x) \in \Ccal$}
\State \Return $\alpha$
\Else
\State $\alpha  \leftarrow \tau_1 \alpha$
\EndIf
\EndWhile
\end{algorithmic}
\end{algorithm}

\begin{algorithm}
\caption{Eigenstep backtracking, modified to stay in $\Ccal$\label{algo:2nd-backtracking}}
\begin{algorithmic}[1]
\State \textbf{Given:} $x\in \calC$, unit-norm $d\in \calE$, $\alpha_{02}>0$, $0<c_2<1/2$, $0<\tau_2<1$.
\State Set $\alpha \leftarrow \alpha_{02}$
\While{true}
\If{$g(x) - g(x + \alpha d) \geq  - c_2 \alpha^2 \langle d, \nabla^2 g(x)[d]\rangle$ \textbf{and} $x + \alpha d \in \Ccal$}
\State \Return $\alpha$
\Else
\State $\alpha  \leftarrow \tau_2 \alpha$
\EndIf
\EndWhile
\end{algorithmic}
\end{algorithm}
\FloatBarrier

\subsection{First-order steps}
We show that small enough gradient steps remain in the set $\calC$ (defined in~\aref{assu:ROI}).
\begin{proposition}
Assume~\aref{assu:ROI} holds with constant $R$ and~\aref{assu:lipschitzh} holds with constant $C_h$. Then, for all $x\in \calC$, if $\beta >\beta_1(x)$, it holds that $ x - t \nabla g(x)$ is in $\calC$ for all $t$ in the interval $[0,t_1(x)]$ where $t_1(x)$ is defined by 
\begin{equation}
 t_1(x) :=\min \!\left(  \sqrt{\dfrac{R}{2C_h}}\dfrac{1}{\norm{\nabla g(x)}}    ,\dfrac{(2 \beta \sigmamin(\D h(x))^2 - \sigma_1(\D h(x)) C_\lambda(x))R}{2C_h \norm{\nabla g(x)}^2}, \dfrac{1}{2\beta \opnorm{\D h(x)}^2} \right),
\label{eq:t_1}
\end{equation}
where $C_\lambda(x) = \opnorm{\D \lambda(x)}$ (Definition~\ref{def:Clambda}).
\label{prop:t1}
\end{proposition}
\begin{proof}
Given $x\in \Ccal$, consider the gradient step $x_t = x - t \nabla g(x)$ for some $t\geq 0$. We wish to find $\tmax>0$ such that $x_t \in \Ccal$ for all $t\in [0,\tmax]$. Using~\aref{assu:lipschitzh}, we have
\begin{align}
h(x_t) &= h(x-t\nabla g(x))\\
 &= h(x) - t \D h(x)[\nabla g(x)] + E\!\left(x, -t \nabla g(x)\right)
\end{align}
where $\norm{E(x,-t\nabla g(x))}\leq C_h \norm{t \nabla g(x)}^2$. Using Equation~\eqref{eq:gradg} gives
\begin{align}
h(x_t) &= h(x) - t\D h(x) \!\left[ \grad_{\M_x} f(x) + 2\beta \D h(x)^*[h(x)] - \D\lambda(x)^* [h(x)]\right] +E(x,-t\nabla g(x)).
\end{align}
Since $\grad_{\M_x} f(x)$ belongs to $ \ker \left( \D h(x)\right)$ by construction, one term cancels:
\begin{align}
h(x_t) &= h(x) - 2\beta t\D h(x) \!\left[   \D h(x)^*[h(x)] \right] + t \D h(x) \!\left[ \D\lambda(x)^* [h(x)]\right] +E(x,-t\nabla g(x))\\
&= \left(\I_m -  2\beta t \D h(x) \circ \D h(x)^* \right) \![h(x)] +  \left( t \D h(x) \circ \D\lambda(x)^* \right)\! [h(x)] + E(x,-t\nabla g(x)).
\end{align}
Let $ \sigma_1 \geq \dots > \sigma_m > 0$ denote the singular values of $\D h(x)$. The eigenvalues of the symmetric operator $\left(\I_m -  2\beta t \D h(x) \circ \D h(x)^* \right) $ are $1-2\beta t \sigma_1^2 \leq \dots \leq 1-2\beta t \sigma_m^2$. All these eigenvalues are smaller than one and are nonnegative provided $0 \leq 1-2\beta t \sigma_1^2$ and $t\geq 0$, or equivalently:
\begin{align}
0 \leq t \leq \dfrac{1}{2 \beta \opnorm{\D h(x)}^2}.
\label{eq:condition_t_1}
\end{align}
Under that assumption, we further find:
\begin{align}
\norm{h(x_t)} &\leq (1-2\beta t \sigma^2_m)\norm{h(x)} + t \sigma_1 \opnorm{\D\lambda(x)^*} \norm{h(x)} + C_h t^2 \norm{\nabla g(x)}^2.
\end{align}
We want to show  $\norm{h(x_t)}\leq R$, which is indeed the case if
\begin{equation}
(1-2\beta t \sigma_m^2)\norm{h(x)} + t \sigma_1 C_\lambda(x) \norm{h(x)} + C_h \norm{\nabla g(x)}^2t^2 \leq R.
\end{equation}
Thus, we seek conditions on $t$ to ensure that the following quadratic inequality in $t$ holds:
\begin{equation}
C_h \norm{\nabla g(x)}^2 t^2 + \left(\sigma_1 C_\lambda(x) - 2 \beta \sigma_m^2 \right) \norm{h(x)} t + \norm{h(x)} - R \leq 0.
\label{eq:Q}
\end{equation}
We branch into two cases. Firstly, consider $\norm{h(x)} \in [R/2, R]$. In this case,~\eqref{eq:Q} holds a fortiori if we remove the independent term $\|h(x)\| - R$ since the latter is nonpositive. By assumption, $\beta > \beta_1(x) = \sigma_1 C_\lambda(x)/2\sigma_m^2$, so the linear term is nonpositive. Therefore, we can upper bound the quadratic by setting $ \norm{h(x)} =R/2$. This shows that
\begin{align}
&C_h \norm{\nabla g(x)}^2 t^2 + \left(\sigma_1 C_\lambda(x) - 2 \beta \sigma_m^2 \right) \norm{h(x)} t  + \norm{h(x)} -R \\ \nonumber
&~~~~~\leq C_h \norm{\nabla g(x)}^2 t^2 + \left(\sigma_1 C_\lambda(x) - 2 \beta \sigma_m^2 \right)\dfrac{R}{2}t.
\end{align}
The above is a convex quadratic with two real roots. It is nonpositive---and~\eqref{eq:Q} is satisfied---if:
\begin{align}
0 \leq t  &\leq \dfrac{ \left(2\beta \sigma_m^2 - \sigma_1 C_\lambda(x)\right)R}{2C_h\norm{\nabla g(x)}^2}.
\label{eq:condition_t_2}
\end{align}
For $\norm{h(x)} \in [0, R/2]$, the linear term in~\eqref{eq:Q} is still nonpositive. Additionally, the constant term of the quadratic is upper bounded by $-R/2$. This establishes
\begin{align}
 C_h \norm{\nabla g(x)}^2 t^2 + \left(\sigma_1 C_\lambda(x) \norm{h(x)} - 2 \beta \sigma_m^2 \norm{h(x)}\right)t  + \norm{h(x)} -R   &\leq C_h \norm{\nabla g(x)}^2 t^2 - \dfrac{R}{2}.
\end{align}
We infer that, for $\norm{h(x)}\in [0, R/2]$, condition~\eqref{eq:Q} is satisfied for
\begin{equation}
0 \leq t\leq \sqrt{\dfrac{R}{2C_h}}\dfrac{1}{\norm{\nabla g(x)}}.
\label{eq:condition_t_3}
\end{equation}
The main claim follows by collecting the conditions in equations~\eqref{eq:condition_t_1},~\eqref{eq:condition_t_2} and~\eqref{eq:condition_t_3}.\end{proof}

We now show that the backtracking in Algorithm~\ref{algo:1st-backtracking} terminates in a finite number of steps and guarantees a sufficient decrease. 
\begin{lemma}[Gradient step decrease]
Take $x\in \Ccal$ and $\beta > \beta_1(x)$. The backtracking procedure in Algorithm~\ref{algo:1st-backtracking} terminates with a step-size $t \geq \tau_1 \min(\underline{\alpha_1}, t_1(x))>0$ where
\begin{equation*}
\underline{\alpha_1} = \min\left(\alpha_{01}, \dfrac{2(1-c_1)}{L_g}\right),
\end{equation*}
with $\alpha_{01}>0$ the initial step size of Algorithm~\ref{algo:1st-backtracking}
and $t_1(x)$ is defined in Equation~\eqref{eq:t_1}.
This guarantees the following decrease:
\begin{equation}
 g(x) -  g(x - t \nabla g(x)) \geq c_1 \tau_1 \min(\underline{\alpha_1}, t_1(x)) \norm{\nabla g(x)}^2.
 \label{eq:decrease-gradient-step}
\end{equation}
\end{lemma}
\begin{proof}
From Proposition~\ref{prop:t1},	we know that $x - \alpha \nabla g(x) $ is in $\Ccal$ for every $0 \leq \alpha \leq t_1(x)$. We proceed to show that the Armijo decrease condition~\eqref{eq:decrease1st} is satisfied for any $0 \leq \alpha \leq \min(t_1(x),\underline{\alpha_1})$. For every $0 \leq \alpha \leq t_1(x)$, the norm of the Hessian of $g$ is bounded by the constant $L_g$ (Equation~\eqref{eq:Lg_M_g}), which implies that $\nabla g$ is $L_g$-Lipschitz continuous on the segment that connects $x$ and $x- t_1(x) \nabla g(x)$. Thus, for all $0\leq \alpha\leq t_1(x)$, we have
\begin{align*}
g(x- \alpha \nabla g(x)) &\leq g(x) + \langle -\alpha\nabla g(x),\nabla g(x)\rangle + \dfrac{L_g}{2}\norm{\alpha \nabla g(x)}^2\\
&= g(x) + \!\left(\dfrac{\alpha L_g}{2}-1\right)\!\alpha \norm{\nabla g(x)}^2.
\end{align*}
This is equivalent to $g(x) - g(x- \alpha \nabla g(x)) \geq \!\left(1-\alpha L_g/2\right)\!\alpha \norm{\nabla g(x)}^2$. For $0\leq \alpha \leq 2(1-c_1)/L_g$, we have $(1-\alpha L_g/2)\geq c_1$. Hence, for $0\leq \alpha \leq \min(t_1(x),\underline{\alpha_1})$, condition~\eqref{eq:decrease1st}, $g(x) - g(x-\alpha \nabla g(x)) \geq c_1 \alpha \norm{\nabla g(x)}^2$ is satisfied. Given $\beta >\beta_1(x)$, one readily checks that $t_1(x)$ is positive. Since $\underline{\alpha_1}$ is also positive, there exists a nonempty interval, $]0,\min(\underline{\alpha_1},t_1(x))]$, where the step size satisfies the Armijo condition and defines a next iterate inside $\calC$. Therefore, Algorithm~\ref{algo:1st-backtracking} returns a step $t$ satisfying $t \geq \tau_1 \min(\underline{\alpha_1}, t_1(x))$. In addition, the Armijo condition gives
\begin{align*}
 g(x) -  g(x - t \nabla g(x)) &\geq c_1 t \norm{\nabla g(x)}^2\\
  &\geq c_1 \tau_1 \min(\underline{\alpha_1}, t_1(x)) \norm{\nabla g(x)}^2.\qedhere
\end{align*}
\end{proof}

\subsection{Second-order steps}
 We begin with a result which guarantees small enough steps stay in $\calC$ when $\nabla g(x)$ is small.
\begin{proposition}\label{prop:t2}
Suppose~\aref{assu:ROI} and~\aref{assu:lipschitzh} hold. Take  $x\in \Ccal$ with $\beta>\max \left\lbrace \beta_1(x), \beta_2(x), \beta_3(x)\right\rbrace$.  Assume that $\norm{\nabla g(x)}\leq \varepsilon_1$ for some $\varepsilon_1 \leq R/2$ so that Proposition~\ref{thm:g-approx-stationary} applies. For any $d\in \calE$ with $\norm{d}=1$, the point $x+td$ is in $\Ccal$ for all $t$ in the interval $[0,t_2(x)]$ with $t_2(x)$ defined by
\begin{equation}
t_2(x) := \left(-\sigma_1(\D h(x)) + \sqrt{\sigma_1(\D h(x)) ^2 +2C_hR}\right)/2C_h.
\label{eq:t2}
\end{equation}
\end{proposition}
\begin{proof}
 Since $\norm{\nabla g(x)}\leq \varepsilon_1$, Proposition~\ref{thm:g-approx-stationary} ensures $\norm{h(x)}\leq \varepsilon_1$. For $t>0$,~\aref{assu:lipschitzh} yields
\begin{align*}
h(x + td) &= h(x) + t \D h(x)[d] + E(x, td),\\
\norm{h(x + td)} &\leq \norm{h(x)} + t\sigma_1 \norm{d} + C_h t^2 \norm{d}^2\\
&\leq \varepsilon_1 + t \sigma_1 + C_h t^2,
\end{align*}
where $\sigma_1$ is the largest singular value of $\D h(x)$. We want to find the values of $t\geq 0$ for which $\varepsilon_1 + t \sigma_1 + C_h t^2 \leq R$. 
The convex quadratic $ t \mapsto C_h t^2 + \sigma_1 t + \varepsilon_1 -R$ has roots $\left(-\sigma_1 \pm \sqrt{\sigma_1 ^2 -4(\varepsilon_1 - R)C_h}\right)/2C_h$, which for $\varepsilon_1 <R$ are real and of opposite signs. 
Hence, the quadratic is nonpositive for all $t$ such that
\begin{equation*}
0 \leq t \leq \left(-\sigma_1 + \sqrt{\sigma_1 ^2 +4|R - \varepsilon_1|C_h}\right)/2C_h.
\end{equation*}
By assumption, $\varepsilon_1 \leq R/2$ and therefore $x+td$ belongs to $\calC$ for all $t$ such that
\begin{equation*}
0 \leq t \leq \left(-\sigma_1 + \sqrt{\sigma_1 ^2 +4C_hR/2}\right)/2C_h.\qedhere
\end{equation*}
\end{proof}

We now show that the backtracking of Algorithm~\ref{algo:2nd-backtracking} terminates in a finite number of steps and guarantees a sufficient decrease.

\begin{lemma}[Eigenstep decrease]
Take $x\in \calC$ and $\beta > \max\left\lbrace \beta_1(x), \beta_2(x), \beta_3(x)\right \rbrace$ with $\norm{\nabla g(x)} \leq \varepsilon_1$ for some $\varepsilon_1 \leq R/2$. Assume there exists a direction $d\in \Ecal$ such that $\norm{d}=1 $, $\langle d, \nabla^2 g(x)[d]\rangle < - \varepsilon_2 $ for some $\varepsilon_2>0$ and $\langle d, \nabla g(x) \rangle \leq 0$. The backtracking procedure in Algorithm~\ref{algo:2nd-backtracking} terminates with a step size $t \geq \tau_2 \min(\alpha_2(x), t_2(x))>0$ where
\begin{equation*}
\alpha_2(x) = \min\left( \alpha_{02}, \dfrac{3|2c_2-1| |\langle d, \nabla^2 g(x)[d]\rangle|}{M_g}\right),
\end{equation*}
with $\alpha_{02}>0$ the initial step size of Algorithm~\ref{algo:2nd-backtracking}
and $t_2(x)$ is defined in Equation~\eqref{eq:t2}. This ensures the following decrease:
\begin{equation}
g(x) - g(x+td) \geq -c_2 \tau_2^2\min(\alpha_2(x), t_2(x))^2 \langle d, \nabla^2 g(x)[d]\rangle.
\label{eq:decrease-eigenstep}
\end{equation}
\end{lemma}
\begin{proof}
From Proposition~\ref{prop:t2}, the point $x+\alpha d $ is in $\Ccal$ for all $0\leq \alpha \leq t_2(x)$. We show that for all $0\leq \alpha \leq \min(\alpha_2(x), t_2(x))$, the decrease condition~\eqref{eq:decrease2nd} is satisfied. 
For every $0 \leq \alpha \leq t_2(x)$, the norm of the third derivative of $g$ is bounded by the constant $M_g$ (Equation~\eqref{eq:Lg_M_g}), which implies that $\nabla^2 g$ is $M_g$-Lipschitz continuous on the segment that connects $x$ and $x+ t_2(x) d$. Thus, for all $0\leq \alpha\leq t_2(x)$, we have
\begin{align*}
g(x +\alpha d) &\leq g(x) + \alpha \langle d, \nabla g(x) \rangle + \dfrac{\alpha^2}{2}\langle d, \nabla^2 g(x)[d]\rangle + \dfrac{M_g}{6}\alpha^3 \norm{d}^3\\
&\leq g(x) + \dfrac{\alpha^2}{2}\langle d, \nabla^2 g(x)[d]\rangle + \dfrac{M_g}{6}\alpha^3\\
&\leq g(x) + \dfrac{\alpha^2}{2}\left( \langle d, \nabla^2 g(x)[d]\rangle + \dfrac{M_g \alpha }{3} \right).
\end{align*}
The sufficient decrease condition~\eqref{eq:decrease2nd}, $g(x) - g(x + \alpha d) \geq  - c_2 \alpha^2 \langle d, \nabla^2 g(x)[d]\rangle$, is satisfied if
\begin{equation*}
-  \dfrac{\alpha^2}{2}\left( \langle d, \nabla^2 g(x)[d]\rangle + \dfrac{M_g \alpha }{3} \right) \geq - c_2 \alpha^2 \langle d, \nabla^2 g(x)[d]\rangle.
\end{equation*}
This is equivalent to
 \begin{align*}
  \langle d, \nabla^2 g(x)[d]\rangle + \dfrac{M_g \alpha }{3}  &\leq 2c_2 \langle d, \nabla^2 g(x)[d]\rangle\\
   \alpha &\leq \dfrac{3(2c_2-1) \langle d, \nabla^2 g(x)[d]\rangle}{M_g}\\
 \alpha &\leq \dfrac{3|2c_2-1| |\langle d, \nabla^2 g(x)[d]\rangle|}{M_g},
\end{align*}
since $c_2<1/2$. Therefore,~\eqref{eq:decrease2nd} is satisfied for all $\alpha \leq \min(\alpha_2(x), t_2(x))$. One readily checks that $t_2(x)$ and $\alpha_2(x)$ are positive. Therefore, there exists a nonempty interval $(0, \min \!\left(\alpha_2(x), t_2(x)\right)]$ where the step-size satisfies the decrease condition~\eqref{eq:decrease2nd} and defines a next iterate inside $\calC$. Therefore the backtracking in Algorithm~\ref{algo:2nd-backtracking} returns a step-size $t$ satisfying $t \geq \tau_2 \min(\alpha_2(x), t_2(x))$ in a finite number of iterations. In addition, the decrease condition~\eqref{eq:decrease2nd} provides
\begin{align*}
 g(x) -  g(x + t d) &\geq -c_2 t^2 \langle d, \nabla^2 g(x)[d]\rangle\\
  &\geq -c_2 \tau_2^2\min(\alpha_2(x), t_2(x))^2 \langle d, \nabla^2 g(x)[d]\rangle. \qedhere
\end{align*}\end{proof}

\begin{remark}
It may seem surprising that $\underline{\alpha_1}$ is a constant and $\alpha_2(x)$ depends on $x$ through the quadratic term $|\langle d, \nabla^2 g(x)[d]\rangle|$. This is a consequence of the way first- and second-order directions are defined. The step-size for a first-order step multiplies the gradient which can vary in norm whereas the step-size in second-order steps always multiplies a unit-norm direction.
\end{remark}

\subsection{Worst-case global complexity}
We are now in a position to give a formal version of our main result: the worst-case complexity of the Gradient-Eigenstep algorithm for problem~\eqref{eq:P}.
\begin{theorem}[Complexity of Algorithm~\ref{algo:gradient-eigenstep}] Consider Problem~\eqref{eq:P} under \aref{assu:ROI}, \aref{assu:boundedMandC}, \aref{assu:lipschitzh}, \aref{assu:x0} and \aref{assu:beta_large_enough}. Let $0 < \varepsilon_1 \leq R/2$ and let $\underline{g}$ be the lower bound of $g$ over the compact set $\Ccal$. Algorithm~\ref{algo:gradient-eigenstep} produces an iterate $x_{N_1} \in \Ccal$ satisfying $\norm{\nabla g(x_{N_1})} \leq \varepsilon_1$ with
\begin{equation}
N_1 \leq \dfrac{g(x_0) - \underline{g}}{ c_1 \tau_1  \min(\underline{\alpha_1}, \underline{t_1})  \varepsilon_1^2},
\end{equation}
where $\underline{t_1} = \min_{x\in \calC} t_1(x)>0$.\nl
Furthermore if $0< \varepsilon_2 < \infty$, Algorithm~\ref{algo:gradient-eigenstep} also produces an iterate $x_{N_2}$ satisfying $\norm{\nabla g(x_{N_2})} \leq \varepsilon_1$ and  $\lambdamin \! \left(\nabla^2 g(x_{N_2}) \right)\! \geq - \varepsilon_2$ with
\begin{equation}
N_2 \leq (g(x_0)- \underline{g}) \!\left[ \min \!\left( c_1 \tau_1 \min(\underline{\alpha_1}, \underline{t_1}) \varepsilon_1^2, c_2 \tau_2^2\min \!\left(   \min\!\left(\alpha_{02}, \dfrac{3|2c_2-1| \varepsilon_2}{M_g}\right), \underline{t_2}\right)^2 \! \varepsilon_2 \right)\!\right]^{-1}\!,
\label{eq:max_iter_2nd_order}
\end{equation}
where $\underline{t_2} = \min_{x\in \calC} t_2(x)>0$.
The iterate $x_{N_1}$ is an $(\varepsilon_1, 2\varepsilon_1)$-FOCP of~\eqref{eq:P} and $x_{N_2}$ is an $(\varepsilon_1, 2\varepsilon_1, \varepsilon_2 + C\varepsilon_1)$-SOCP of~\eqref{eq:P}, where $C$ is defined in Corollary~\ref{corollary:approx_stationarity_hessian-global}. 
\label{thm:complexity}
\end{theorem}
\begin{proof}
We first show that the constants  $\underline{t_1} = \min_{x\in \calC} t_1(x)$ and $\underline{t_2} = \min_{x\in \calC} t_2(x)$ are positive. Recall from Equation~\eqref{eq:t_1} that
\begin{equation*}
 t_1(x) =\min \!\left(  \sqrt{\dfrac{R}{2C_h}}\dfrac{1}{\norm{\nabla g(x)}}    ,\dfrac{(2 \beta \sigmamin(\D h(x))^2 - \sigma_1(\D h(x)) C_\lambda(x))R}{2C_h \norm{\nabla g(x)}^2}, \dfrac{1}{2\beta \opnorm{\D h(x)}^2} \right)
\end{equation*}
One readily checks that $t_1(x)>0$ for all $x\in \calC$.
The first term, $\sqrt{R/2C_h}/\norm{\nabla g(x)}$ is positive since $\nabla g$ is continuous over $\calC$ and $\calC$ is compact. Using that $\beta>\bar{\beta_1}$ (\aref{assu:beta_large_enough}), the numerator of the second term is positive and bounded away from zero for all $x\in \calC$. 
Using compactness of $\calC$ and smoothness of $h$, the quantity $\opnorm{\D h(x)}$ is upper bounded over $\calC$ and therefore $1/2\beta\opnorm{\D h(x)}^2$ is bounded away from zero over $\calC$. We note that $t_1$ is a continuous function of $x$ which is positive for all $x$ in the compact set $\calC$. Therefore,  $\min_{x\in \calC} t_1(x)$ is attained at a point in $\calC$ and $\underline{t_1}>0$. A similar process shows that $\underline{t_2}>0$. The function $t_2(x) = \left(-\sigma_1(\D h(x)) + \sqrt{\sigma_1(\D h(x)) ^2 +2C_h R}\right)/2C_h$ is continuous over $\calC$. We also note that $t_2(x)>0$ for all $x\in \calC$ since the constants $R$ and $C_h$ are positive as a consequence of~\aref{assu:ROI} and~\aref{assu:lipschitzh} respectively. 

For every iterations $k$  where a first-order step is performed, one has $\norm{\nabla g(x_k)} > \varepsilon_1$, while for second-order steps $\langle d, \nabla^2g(x_k)[d]\rangle < - \varepsilon_2$. Therefore, Equation~\eqref{eq:decrease-gradient-step} gives the following decrease for first-order steps:
\begin{align}
 g(x_k) -  g(x_{k+1}) &\geq c_1 \tau_1 \min(\underline{\alpha_1}, t_1(x_k)) \norm{\nabla g(x_k)}^2 \nonumber\\
 &\geq c_1 \tau_1 \min(\underline{\alpha_1}, \underline{t_1}) \varepsilon_1^2,
\end{align}
where $\underline{t_1} = \min_{x\in \calC} t_1(x)>0$, as shown above. The decrease for second-order steps follows from Equation~\eqref{eq:decrease-eigenstep}, that is,
\begin{align}
g(x_k) - g(x_{k+1})  &\geq -c_2 \tau_2^2\min(\alpha_2(x_k), t_2(x_k))^2 \langle d, \nabla^2 g(x)[d]\rangle\nonumber\\
&\geq c_2 \tau_2^2\min\left( \min\left(\alpha_{02}, \dfrac{3|2c_2-1| |\langle d, \nabla^2 g(x_k)[d]\rangle|}{M_g}\right), t_2(x_k)\right)^2  \varepsilon_2\nonumber\\
&\geq c_2 \tau_2^2\min\left( \min\left(\alpha_{02}, \dfrac{3|2c_2-1| \varepsilon_2}{M_g}\right), \underline{t_2}\right)^2  \varepsilon_2,
\end{align}
where $\underline{t_2} = \min_{x\in \calC} t_2(x)>0$, as shown above.
Since $\Ccal$ is compact (\aref{assu:boundedMandC}) and $g$ is continuous on $\calC$, let $\underline{g} := \min_{x\in\calC} g(x)>-\infty$. Consider the case $\varepsilon_2<\infty$. For any $K\geq 0$, we have
\begin{align}
g(x_0) - \underline{g}& \geq \sum_{k=0}^K g(x_k) - g(x_{k+1}) \\
&\geq K  \min\left( c_1 \tau_1 \min(\underline{\alpha_1}, \underline{t_1}) \varepsilon_1^2, c_2 \tau_2^2\min\left(   \min\left(\alpha_{02}, \dfrac{3|2c_2-1| \varepsilon_2}{M_g}\right), \underline{t_2}\right)^2  \varepsilon_2 \right).\label{eq:K_2nd_order}
\end{align}
Given the definition of $N_2$, Equation~\eqref{eq:K_2nd_order} tells us that $K\leq N_2$. Hence, if more than $N_2$ iterations are performed, it must be that a point where $\norm{\nabla g(x)} \leq \varepsilon_1$ and $\lambdamin(\nabla^2 g(x)) \geq -\varepsilon_2$ has been encountered. In the case $\varepsilon_2 = \infty$, no second-order step is performed, which simplifies as follows:
 \begin{align}
g(x_0) - \underline{g}& \geq \sum_{k=0}^K g(x_k) - g(x_{k+1}) \nonumber\\
&\geq K  c_1 \tau_1  \min(\underline{\alpha_1}, t_1(x_k))  \norm{\nabla g(x_k)}^2\nonumber\\
&\geq K  c_1 \tau_1  \min(\underline{\alpha_1}, \underline{t_1})  \varepsilon_1^2.
\end{align}
The fact that $x_{N_1}$ and $x_{N_2}$ are respectively $(\varepsilon_1, 2\varepsilon_1)$-FOCP and $(\varepsilon_1, 2\varepsilon_1, \varepsilon_2 + C\varepsilon_1)$-SOCP of~\eqref{eq:P} follows from Proposition~\ref{thm:g-approx-stationary} and Proposition~\ref{thm:approx_stationarity_hessian}.
\end{proof}
\section{Estimating the penalty parameter}
\label{sec:plateau-scheme}
The previous section establishes convergence results under the assumption that the penalty parameter $\beta$ is large enough to satisfy~\aref{assu:beta_large_enough}. In practice, it is rarely possible to know whether this assumption is satisfied. Therefore, this section outlines a scheme which estimates a suitable value for $\beta$. A simple strategy consists in letting Algorithm~\ref{algo:gradient-eigenstep} run for a fixed number of iterations using a given value of $\beta$, and to repeat this process with increasingly larger values of $\beta$ until convergence of Algorithm~\ref{algo:gradient-eigenstep} to an approximate critical point is achieved. We refer to such strategy as a plateau scheme. Along the way, if the algorithm encounters a point $x\in \calC$ such that $\beta \leq \max\{\beta_1(x), \beta_2(x), \beta_3(x)\}$ (Definition~\ref{def:beta123}), $\beta$ is increased and the procedure is restarted. It is tedious but not difficult to show that the complexity of such a plateau scheme is only a logarithmic factor worse than the complexity of Algorithm~\ref{algo:gradient-eigenstep} -- with a fixed value of $\beta$ that satisfies~\aref{assu:beta_large_enough}.
\begin{theorem}
Under~\aref{assu:ROI},~\aref{assu:boundedMandC},~\aref{assu:lipschitzh} and \aref{assu:x0}, a plateau scheme 
returns an $(\varepsilon_1,2 \varepsilon_1, \varepsilon_2 + C\varepsilon_1)$-SOCP in at most $\bigO\left(\max\left\lbrace \varepsilon_1^{-2}, \varepsilon_2^{-3}\right\rbrace   \max\left\lbrace \log_\gamma \varepsilon_1^{-2}, \log_\gamma \varepsilon_2^{-3}\right\rbrace \right)$ gradient steps and eigensteps on the function $g$, where $\gamma>1$ is the growth rate of $\beta$ between two plateaus and $C$ is defined in Corollary~\ref{corollary:approx_stationarity_hessian-global}.
\end{theorem}
\begin{arxiv}
For a complete definition of a plateau scheme and proof of the above result, see Appendix~\ref{appendix:plateau}.
\end{arxiv}

\section{Conclusion and discussion}

In this work, we consider a penalty function (Fletcher's augmented Lagrangian) for optimization under smooth equality constraints. We establish connections between its approximate critical points and the approximate critical points of the original constrained problem~\eqref{eq:P}. We also highlight that various definitions of approximate second-order critical points for equality constraints appear in the literature. Therefore, we propose a definition of approximate criticality which has a natural geometric interpretation and extends Riemannian optimality conditions to points near the feasible set.

We present Algorithm~\ref{algo:gradient-eigenstep}, which is shown to reach approximate second-order critical points of~\eqref{eq:P} in at most $\bigO(\varepsilon^{-3})$ iterations. The only other work to date which achieved this optimal rate for an infeasible method is~\citep{cartis2019optimality}, where the definition of approximate critical point is markedly different. 
Finally, we describe how Algorithm~\ref{algo:gradient-eigenstep} can be modified to achieve a local quadratic convergence rate. 

 The main drawback of our approach, is the necessity to identify a set $\calC$, where the differential of the constraint is nonsingular, in order to run the algorithm. Similar smoothness assumptions are made in related works which provide a worst-case complexity analysis~\citep{cifuentes2019polynomial,xie2021complexity}. It would nonetheless be worthwhile to go beyond such assumptions.

Fletcher's augmented Lagrangian may be considered impractical in view of the linear system that must be solved at each iteration to evaluate the multipliers $\lambda(x)$. However, recent works show that it can still lead to the design of efficient algorithms and this work further reinforces the theoretical appeal of Fletcher's augmented Lagrangian. Directions of future research also emerge. Consider a smooth function $\hat \lambda \colon \calE \to \Rm$ which coincides on $\M$ with the function $\lambda(x) = \left( \D h(x)^*\right)^\dagger [\nabla f(x)] $ considered in this work. This choice of multipliers defines a corresponding function $\hat g(x) = \calL_\beta(x, \hat \lambda(x))$, a variant of the $g$. Recent works~\citep{gao2019parallelizable,xiao2020class,xiao2021solving} show that minimizing the function $\hat g$ yields efficient algorithms for a particular choice of $\hat \lambda$ on the Stiefel manifold. Is there a way to generalize this concept to other manifolds? What theoretical guarantees can we hope to keep by using $\hat \lambda(x)$ instead of $\lambda(x)$? Exploring this could yield more practical Lagrangian-based infeasible methods to solve constrained optimization problems with underlying smoothness.

\bibliographystyle{apalike}
\bibliography{fletcher.bib}
 \appendix

 \section{Related work}
\label{sec:literature}

\subsection*{Complexity in constrained optimization}

The study of complexity in optimization has been very active in recent years, both for constrained and unconstrained problems. The field focuses on giving guarantees on the worst-case number of iterations an algorithm requires to achieve a predetermined termination criterion. The first results dealt with the unconstrained case, where $\M= \calE$. Among others,~\citet{Nesterov2004} shows that for Lipschitz differentiable $f$, gradient descent with an appropriate step size requires at most $\bigO(\varepsilon^{-2})$ iterations to find a point which satisfies $\norm{\nabla f(x)}\leq \varepsilon$. This is sharp~\citep{cartis2010complexity}, meaning that there exist functions for which gradient descent requires that many iterations.\footnote{Note that if Hessian information is available and the Hessian of $f$ is Lipschitz continuous, cubic regularization has a $\bigO(\varepsilon^{-3/2})$ complexity for approximate first-order critical points~\citep{Nesterov2006}. Using derivatives of higher order can further improve the rate of regularization methods~\citep{birgin2017worst}.} \citet{Cartis2012} further show that a point which satisfies both $\norm{\nabla f(x)}\leq \varepsilon$ and $\lambdamin(\nabla^2 f(x))\geq - \varepsilon$ can be found in $\bigO(\varepsilon^{-3})$ iterations using a cubic regularization method. This bound is also sharp.

The study of the constrained case adds some difficulties. This paper focuses on problems with equality constraints and leaves aside inequality constraints. For complexity bounds of constrained problems, the sharpness of the unconstrained bounds carries over to the constrained case. That is, the best worst-case bounds achievable for the constrained case are also $\bigO(\varepsilon^{-2})$ and $\bigO(\varepsilon^{-3})$ for first- and second-order points using first- and second-order methods respectively. 

Under~\aref{assu:ROI}, problem~\eqref{eq:P} is defined over a smooth manifold. For some known manifolds, such as those described in~\citep{Absil2008}, Riemannian optimization offers an elegant and efficient way to solve constrained optimization problems. \citet{zhang2016complexity,bento2017iteration,boumal2019global} showed that some Riemannian optimization algorithms have the same worst-case bounds as their unconstrained counterparts. That is, under a Lipschitz smoothness assumption, Riemannian gradient descent with an appropriate step size finds a point which satisfies $\norm{\grad_\M f(x)}\leq \varepsilon$ in $\bigO(\varepsilon^{-2})$ iterations. Similarly, a Riemannian trust-region algorithm finds a point which satisfies $\norm{\grad_\M f(x)}\leq \varepsilon$ and $\lambdamin \!\left( \Hess_\M f(x)\right)\! \geq -\varepsilon \Id$ in $\bigO(\varepsilon^{-3})$ iterations. 

Riemannian optimization methods are applicable to manifolds $\M$	 provided that one is able to compute retractions and generate a feasible sequence of iterates. This is sometimes impossible or too expensive computationally. This prompts the use of infeasible methods to solve~\eqref{eq:P}, which are the focus of this paper. 

Several different notions of approximate criticality for~\eqref{eq:P} are in use in the literature. We review and compare them now, together with existing algorithmic guarantees to find such points. Among those that cover approximate second-order critical points, the rates are either not optimal (worse than $\bigO(\varepsilon^{-3})$), or they rely on an unusual notion of criticality. 

For a non-empty, closed convex set $\mathcal{F}$, \citet{cartis2019optimality} consider the problem $\min_{x\in \mathcal{F}} f(x)$ such that $h(x) =0$, which is a problem class more general than~\eqref{eq:P}. They propose a two-phase algorithm which finds approximate first-, second- and even third-order critical points. The first phase of their algorithm attempts to find an approximately feasible point. The second phase minimizes the cost function while tracking infeasibility and staying close to the feasible set. This approach is not particularly efficient in practice, but yields optimal complexity rates for finding approximately critical points, which for first- and second-order are respectively $\bigO(\varepsilon^{-2})$ and $\bigO(\varepsilon^{-3})$ iterations. Their notion of criticality is unusual but has the advantage of generalizing to optimality beyond second order. Considering the merit function
\begin{equation}
 \mu(x, t) := \dfrac{1}{2} \norm{r(x, t)}^2 := \dfrac{1}{2}\norm{\begin{pmatrix}
h(x)\\
f(x) - t
\end{pmatrix}}^2,
\end{equation}
 an approximate second-order critical point $x\in \Ecal$  is defined as satisfying
\begin{align}
\phi_{\mu, j}^{\Delta}(x,t)\leq \varepsilon \Delta^j \norm{r(x,t)} \txt{ for } j = 1, 2
\label{eq:cartis-critical}
\end{align}
where 
\begin{equation}
\phi^\Delta_{\mu, j}:= \mu(x,t) - \min_{\substack{d\in \Ecal \\
\norm{d}\leq \Delta }} T_{\mu, j}(x,d),
\end{equation}
is the largest feasible decrease of the $j$th order Taylor model $T_{\mu, j}(x,d)$ achievable at distance at most $\Delta$ from $x$. 

\citet{cifuentes2019polynomial} tackle the problem of solving semi-definite programs using the Burer-Monteiro factorization.  They adapt the two-phase algorithm from~\citep{cartis2019optimality} so that the target points satisfy the following notion of criticality. For $\gamma>0$ and $\bm{\varepsilon} = (\varepsilon_0, \varepsilon_1, \varepsilon_2)$, they define a point $x\in \calE$ as $(\bm{\varepsilon}, \gamma)$-approximately feasible approximately 2-critical (AFAC) if there exists $\lambda \in \Rm$ such that: 
\begin{multline}
\norm{h(x)} \leq \varepsilon_0\text{, }  \norm{\nabla_{x} \calL(x,\lambda)} \leq\varepsilon_1,  \text{ and }  u\transpose \nabla^2_{xx}L(x, \lambda) u  \geq - \varepsilon_2  \text{ for all } u \text{ of unit norm} \\ \text{ such that } \norm{\D h(x)[u]} \leq \gamma.
 \label{eq:Cifuentes-critical}
\tag{AFAC}
\end{multline}
The set of directions $u\in \calE$ that satisfy $\norm{\D h(x)[u]}\leq \gamma$ for some $\gamma >0$ includes the tangent space defined by $\D h(x)[u]=0$. Therefore, the condition~\eqref{eq:Cifuentes-critical} implies the conditions~\eqref{eq:lagrange-eps-focp} and~\eqref{eq:lagrange-eps-socp} defined below, but the converse is not true. Under Assumptions~\aref{assu:ROI} and~\aref{assu:x0}, along with uniform boundedness and Lipschitz continuity of $f$, $h$ and their derivatives on $\calC$,~\citet{cifuentes2019polynomial} show that an~\eqref{eq:Cifuentes-critical} point can be found in $\bigO\left( \max\left\lbrace \varepsilon_0^{-2}\varepsilon_1^{-2}, \varepsilon_0^{-3}\varepsilon_2^{-3}\right\rbrace\right)$ iterations. The smoothness and initialization assumptions made are mostly equivalent to the ones made in this work. 
They point out that in adapting the two-phase algorithm from~\citep{cartis2019optimality} to guarantee~\eqref{eq:Cifuentes-critical} points, a factor $\varepsilon_0^{-1}$ is lost in the complexity. 

\subsection*{Complexity of augmented Lagrangian methods}

A recent point of interest in the literature has been the study of complexity for algorithms that belong to the family of augmented Lagrangian methods (ALM). These methods have always been popular, with good practical results, but worst-case complexity results are lacking. Augmented Lagrangian methods minimize $\calL_\beta(x,\lambda)$ by updating the variables $x\in \calE$ and $\lambda\in \Rm$ alternatively. At iteration $k$, the subproblem to update $x$ usually involves finding approximate minimizers of $\calL_\beta(\cdot, \lambda_k)$, while the multipliers are updated using the first-order step $\lambda_{k+1} = \lambda_k  -\beta h(x_k)$. The penalty parameter $\beta$ is also typically increased throughout the iterations~\citep{birgin2014practical}.

\citet{xie2021complexity} analyse a proximal ALM, and suggest to solve the subproblems using a Newton-conjugate gradient algorithm from~\citep{royer2020newton}. For this second-order algorithm, they show a total iteration complexity to reach approximate first- and second-order critical points of $\bigO(\varepsilon^{-11/2})  $ and $\bigO(\varepsilon^{-7})$. When $h$ is linear, their guarantees match the best known result of $\bigO(\varepsilon^{-3})$ total iterations for second-order points. These results require the initial iterate to satisfy $\norm{h(x_0)}^2 \leq \min(C_0/\rho,1)$ for some constant $C_0>0$ and $\rho$ that increases as $\varepsilon$ and $\sigmabar$ decrease. This condition is difficult to verify in practice, when there is no simple way to generate a feasible point. Admittedly, it can also be difficult to satisfy our condition~\aref{assu:x0} in general. The advantage of~\aref{assu:x0} is that generating an initial iterate in $\calC$ does not depend on $\varepsilon$ but only on the function $h$. \citet{xie2021complexity} further require that for some $\rho_0\leq 0$, the function $f(x) + \dfrac{\rho_0}{2}\norm{h(x)}^2$ has compact level sets, and also that $f$ be upper-bounded on the set $\{x\in \calE: \norm{h(x)}\leq 1\}$. 

\citet{he2023newton} improved upon the rates of~\citet{xie2021complexity} using a similar regularized augmented Lagrangian method. They show a total iteration complexity of $\bigO(\varepsilon^{-7/2})$ to reach $(\varepsilon, \sqrt{\varepsilon})$-critical points. They also report a rate of $\bigO(\varepsilon^{-11/2})$ which does not require any constraint qualification to hold. Assuming a constraint qualification, they report a total iteration complexity of $\bigO\left(\varepsilon_1^{-2} \max\left\lbrace\varepsilon_1^{-2}\varepsilon_2, \varepsilon_2^{-3}\right\rbrace\right)$, which gives $\bigO(\varepsilon^{-5})$ for $\varepsilon_1 = \varepsilon_2 = \varepsilon$ and $\bigO(\varepsilon^{-7/2})$ for $(\varepsilon_1,\varepsilon_2) =(\varepsilon, \sqrt{\varepsilon})$. Without any constraint qualification, they report a total iteration complexity of $\bigO\left(\varepsilon_1^{-4} \max\left\lbrace\varepsilon_1^{-2}\varepsilon_2, \varepsilon_2^{-3}\right\rbrace\right)$, which gives $\bigO(\varepsilon^{-7})$ for $\varepsilon_1 = \varepsilon_2 = \varepsilon$ and $\bigO(\varepsilon^{-11/2})$ for $(\varepsilon_1,\varepsilon_2) =(\varepsilon, \sqrt{\varepsilon})$.

Both~\citet{xie2021complexity,he2023newton} consider optimality conditions based on the Lagrangian function, which generalize well-known optimality conditions under constraints to non-feasible points. The point $x\in\calE$ is approximately second-order critical if there exists $\lambda \in \Rm$ such that
\begin{align}
\norm{h(x)} &\leq \varepsilon_0, & \norm{\nabla_x \calL(x, \lambda)}&\leq \varepsilon_1 
\label{eq:lagrange-eps-focp}
\end{align}
and
\begin{equation}
\inner{\nabla^2_{xx} \mathcal{L}(x,\lambda)[v]}{v} \geq -\varepsilon_2 \norm{v}^2 \textrm{  for all  } v\in \calE \textrm{ such that  }\D h(x)[v]=0.
\label{eq:lagrange-eps-socp}
\end{equation}
Note that the conditions~\eqref{eq:lagrange-eps-focp} and~\eqref{eq:lagrange-eps-socp} are not equivalent to~\eqref{eq:SOCP}. If $x\in \calE$ is an~\eqref{eq:SOCP}, then it satisfies~\eqref{eq:lagrange-eps-focp} and~\eqref{eq:lagrange-eps-socp} with multipliers $ \lambda(x)\in \Rm$. However, Equations~\eqref{eq:lagrange-eps-focp} and~\eqref{eq:lagrange-eps-socp} do not imply~\eqref{eq:SOCP}. As a counter-example, in $\calE= \Rn$ take $\varepsilon_1 = \varepsilon_2 = 1/2$ with the functions $h(x) = \norm{x}^2-1$ and $f(x) = \inner{x}{w}$ for some $w\in \calE$, $\norm{w} = 1$. The point $w\in \calE$ satisfies~\eqref{eq:lagrange-eps-focp} and~\eqref{eq:lagrange-eps-socp} with multiplier $\lambda = 1/4$, even though it is the \emph{maximizer} of the function $f$ on the sphere, which is undesirable. However, $\lambdamin (\Hess_{\M} f(w)) = -1$ hence $w$ is not an~\eqref{eq:SOCP}, which one should expect as it does not approximately minimize $f$ on the sphere. The values of $\varepsilon_1,\varepsilon_2$ in the example are not typical of an optimization algorithm, but it is possible to scale $f$ and $h$ by a constant to retain the conclusion while making $\varepsilon_1,\varepsilon_2$ small. Thus, the two notions only meet if $\varepsilon$ is smaller than some unknown threshold.

\citet{grapiglia2021complexity} provide a worst-case complexity analysis for an augmented Lagrangian method that can  be applied to both equality and inequality constraints. For $h(x_0)=0$, the number of outer iterations to reach a first-order critical point is $\bigO\Big(\varepsilon^{-2/(\alpha-1)}\Big)$ where $\alpha>1$ defines the rate of increase of the penalty parameter $\beta$, i.e., $\beta_{k+1} = \max\{(k+1)^\alpha, \beta_k\}$. In this way, increasing $\alpha$ worsens the conditioning of the subproblems and increases the inner iteration count, which is not included in the bound above. The definition of first-order critical point simplifies to~\eqref{eq:lagrange-eps-focp} for problems with equality constraints only. Under the additional assumption that the penalty parameters $\beta_k$ stay bounded as $k\to \infty$,  the outer complexity is improved to $\bigO(\log(\varepsilon^{-1}))$. It is debatable whether the assumption that the penalty parameters remain bounded is reasonable in practice, as increasing the penalty parameters is often critical to ensure convergence. 

\citet{birgin2019complexity} analyse the outer iteration complexity of the popular optimization software \textsc{Algencan} introduced in~\citep{andreani2008augmented}. This software was designed with practical efficiency in mind. It handles both equalities and inequalities. It is also safeguarded, meaning that upper and lower bounds are imposed on the multipliers. That work also shows an outer iteration complexity of  $\bigO(\log(\varepsilon^{-1}))$ in the case of bounded penalty parameters.

In~\citep{conn1991globally}, the authors present a classical augmented Lagragian algorithm, where the usual first-order update for the multipliers $\lambda_{k+1} = \lambda_k - \beta h(x)$ can be replaced by the least-squares update~\eqref{eq:lambda}, namely, $\lambda_{k+1} = \lambda(x_{k+1})$. They show that if a limit point $x^*$ of the algorithm is feasible, then it is first-order critical for~\eqref{eq:P} with multipliers $\lambda(x^*)$ (Equation~\eqref{eq:riemannian-feasible-focp}).

\citet{sahin2019inexact} study the problem $\min_{x\in\R^n} f_1(x) + f_2(x)$ such that $F(x)=0$ where $f_1$ is nonconvex and smooth, $f_2$ is proximal-friendly and convex and $F\colon \R^n \to \R^m$ is a nonlinear operator. The potential nonsmoothness of $g$ extends the range of applications compared to~\eqref{eq:P}. For an augmented Lagrangian algorithm, they report bounds of $\bigO(\varepsilon^{-3})$ and $\bigO(\varepsilon^{-5})$ outer iterations to find approximate first- and second-order critical points of the augmented Lagrangian, where the subproblems are solved with a first- or second-order solver respectively.

\subsection*{Infeasible optimization methods for Riemannian manifolds}

Optimization under orthogonality constraints appears in a number of applications and is an active area of research. Riemannian optimization methods can be used on the Stiefel manifold $\St(n,p)= \left\lbrace X\in \R^{n\times p}: X\transpose X = \I_n \right\rbrace$ and the orthogonal group $\mathrm{O}(p) = \left \lbrace X\in \R^{p\times p}: X\transpose X = \I_p\right \rbrace$. These algorithms requires to perform an orthogonalization procedure, known as a retraction, at every step throughout the optimization process. When $p$ is small compared to $n$, fast retractions are available for the Stiefel manifold. However, when $p$ is large, computing these retractions is often the computational bottleneck~\citep{gao2019parallelizable}. This has prompted the search for retraction-free algorithms to deal with orthogonality constraints. Our use of Fletcher's augmented Lagrangian is partially inspired by~\citet{gao2019parallelizable}. Those authors propose an algorithm specific to the Stiefel manifold. The algorithm is a primal-dual scheme, which updates alternatively the variable $x$ and the multipliers $\lambda$. The primal update is obtained from approximately minimizing $\calL_\beta (x,\lambda)$ over $x$, while $\lambda$ is updated using a simplified version of formula~\eqref{eq:lambda}, which we call $\hat \lambda(\cdot)$. Consider $\hat g(x)=\calL_\beta(x,\hat\lambda(x))$ the penalty where the least-squares multipliers $\lambda(\cdot)$ are replaced by the approximation $\hat \lambda(\cdot)$.  In~\citep{xiao2020class}, the authors further study the penalty function $\hat g(x)$ for the Stiefel manifold. Recognizing that this function is unbounded below on $\R^{n\times p}$, they add an artificial box constraint around Stiefel to prevent divergence and develop a second-order method to minimize $\hat g(x)$ with asymptotic convergence results.
\citet{ablin2022fast} also present a retraction-free algorithm on the orthogonal group. Their landing algorithm is an infeasible method which converges to an orthogonal matrix through the minimization of a purposefully constructed potential energy function. They report a speed up over classical retraction-based methods on some large-scale problems.
\citet{schechtman2023orthogonal} have proposed a recent follow-up work on a first-order method which extends~\citep{ablin2022fast} to the general manifold $\calM$.

In the recent works~\citep{liu2020riemannianconstraints,jia2021augmented}, the authors consider a framework where, conceptually, part of the constraints are easy to project onto, while the other constraints are more difficult to handle, as can be the case for the constraints of~\eqref{eq:P}. Mixed approaches are proposed where the easy constraints are dealt with in a Riemannian-like fashion, while the other constraints are penalized with an augmented Lagrangian function. 

Table~\ref{table:review} on page~\pageref{table:review}	presents a list of works that have significant theoretical results for a problem definition similar to~\eqref{eq:P}. The table shows whether they consider second-order critical points, how the target points are defined and the results in terms of global complexity and local rate of convergence.

\subsection*{Fletcher's augmented Lagrangian}

Around the same year that augmented Lagrangian methods came about, the penalty function we use in this paper---Fletcher's augmented Lagrangian—was introduced in~\citep{fletcher1970class}. The penalty function is introduced in its general form as
\begin{align*}
f(x) - \inner{h(x)}{\lambda(x)} + \dfrac{1}{2} \inner{h(x)}{Qh(x)},
\end{align*}
where $Q$ is a positive definite matrix. We study the natural choice $Q = \beta \Id$. In the original work, some fundamental properties of the function were established, most notably, connecting critical points of $f$ on $\calM$ and critical points of $g$.  \citet[section 4.3.2]{bertsekas1982constrained} also covers properties of $g$, as an exact penalty functions that depend only on $x$ and does not use a variable for the Lagrange multipliers. These properties are covered in Section~\ref{sec:fletcher-alm}, where we extend them to situations with approximate critical points of first- and second-order.

\citet{di1986exact,pillo1994exact} present algorithms with local convergence analyses that rely on Fletcher's augmented Lagrangian.
\citet{di1989exact} define the Lagrange mutlipliers as
\begin{align}\label{eq:regularized_lambda}
\lambda_\gamma(x) = \argmin_{\lambda \in \R^m} \norm{\D h(x)^* [\lambda] - \nabla f(x)}^2_2 + \dfrac{\gamma}{2}\norm{\lambda}^2_2 ,
\end{align}
for some $\gamma >0$. This regularization ensures that the multipliers $\lambda_\gamma(x)$ are well defined even when $\D h(x)$ is singular. The regularized least-squares problem may also be useful when $\D h(x)$ has full rank but is close to deficiency. In those cases the least-squares problem becomes ill-conditioned, which may leads to inaccurate solutions.

\citet{fletcher1970class} proposed different choices for the matrix $Q$, including the penalty 
\begin{align}
\phi_\beta(x) = f(x) - \inner{h(x)}{\lambda(x)} +\beta \inner{h(x)}{\left( \D h(x)^* \D h(x)\right)^{-1} h(x)},
\end{align}
which is studied in ~\citep{estrin2020implementing}. They present a way to compute $g(x)$, $\nabla g(x)$ and approximations of Hessian-vector products $\hess g(x) v$ which only rely on solving least-square linear systems. \citet[Chap. 15]{rapcsak1997smooth} studies the modified Lagrangian function $\calL(x,\lambda(x))$ and considers the properties of its critical points on a subset of $\calM$ that is geodesically convex.

\section{Proof of Proposition~\ref{prop:feasibleC}} 
 \label{sec:proof_feasibleC}
\begin{proof}
Define $\varphi(x) = \dfrac{1}{2}\norm{h(x)}^2$ and take any $x_0 \in \calC  = \lbrace x\in \calE: \varphi(x) \leq R^2/2\rbrace$. Consider the following differential system:
\begin{equation}
\left\lbrace
\begin{aligned}
\dfrac{\mathrm{d}}{\mathrm{d} t}x(t) & = - \nabla \varphi(x(t)) \\
x(0) & = x_0.
\end{aligned}
\right.
\label{eq:flow}
\end{equation}
The fundamental theorem of flows~\citep[Theorem A.42]{lee2018introduction} guarantees the existence of a unique maximal integral curve starting at $x_0$ for~\eqref{eq:flow}. Let $z( \cdot )\colon I \to \calE$ denote this maximal integral curve and $T>0$ be the supremum of the interval $I$ on which $z(\cdot)$ is defined. 
We rely on the Escape Lemma~\citep[Lemma A.43]{lee2018introduction} to show that $z(t)$ is defined for all times $t\geq 0$. For $t< T$, we write $\ell = \varphi \circ z$ and find 
\begin{align}
\ell'(t) &= \D \varphi (z(t))\left[\ddt z(t)\right] = \inner{\nabla \varphi (z(t))}{\ddt z(t)}\\
  &= - \norm{\nabla \varphi (z(t))}^2 \\
  &= - \norm{ \D h(z(t))^*[h(z(t))]}^2 \leq 0.
\end{align}
This implies that $z(t) \in \calC$ for all $0\leq t < T$. We show that the trajectory $z(t)$ has finite length. To that end, we note that 
\begin{align}\label{eq:PL_h}
\dfrac{1}{2}\norm{\nabla \varphi(x)}^2 = \dfrac{1}{2} \norm{\D h(x)^* [h(x)] }^2 \geq \sigmabar^2 \dfrac{1}{2} \norm{h(x)}^2 =  \sigmabar^2 \varphi(x),
\end{align}
for all $x\in \calC$. The length of the trajectory from time $t=0$ to $t=T$ is bounded as follows, using a classical argument~\citep{lojasiewicz1982trajectoires}:
\begin{align}
\int_0^T \norm{\ddt z(t)}\mathrm{d} t &= \int_0^T \norm{- \nabla \varphi(z(t))}\mathrm{d} t \nonumber\\
 &= \int_0^T \dfrac{ \norm{\nabla \varphi(z(t))}^2}{ \norm{\nabla \varphi(z(t))}} \mathrm{d} t
\nonumber\\
 &= \int_0^T \dfrac{ \inner{- \nabla \varphi(z(t))}{\ddt z(t)}}{ \norm{\nabla \varphi(z(t))}} \mathrm{d} t\nonumber\\
  &= \int_0^T \dfrac{ - (\varphi \circ z)'(t)}{ \norm{\nabla \varphi(z(t))}} \mathrm{d} t\nonumber\\
 &\leq  \int_0^T \dfrac{ - (\varphi \circ z)'(t)}{\sigmabar \sqrt{2(\varphi \circ z)(t)}} \mathrm{d} t\nonumber\\
 &= \dfrac{-\sqrt{2}}{\sigmabar} \left[ \sqrt{\varphi(z(T))} - \sqrt{\varphi(z(0))}\right]\nonumber\\
 &\leq \dfrac{\sqrt{2 \varphi(z(0))}}{\sigmabar}.\label{eq:finite_length}
\end{align}
The length is bounded independently of $T$ and therefore the flow has finite length. The Escape Lemma states that for a maximum integral curve $z(\cdot) \colon I \to \calE$, if $I$ has a finite upper bound, then the curve $z(\cdot)$ must be unbounded. Since $z(\cdot)$ is contained in a compact set by~\eqref{eq:finite_length}, the converse ensures that the interval $I$ does not have a finite upper bound and therefore, $I=\R_+$. 
Since the trajectory $z(t)$ is bounded for $t\geq 0$, it must have an accumulation point $\bar{z}$. From~\aref{assu:ROI}, we have $\sigmamin(\D h(z(t)) \geq \sigmabar>0$ for all $t \geq 0$. This gives the bound $\ell'(t) \leq - \sigmabar^2 \norm{h(z(t))}^2 = -2\sigmabar^2 \ell(t)$. Gronwall's inequality then yields
\begin{align}
\ell(t) \leq \varphi(x_0) e^{-2\sigmabar^2 t}. 
\end{align}
Therefore $\ell(t) \to 0 $ as $t \to \infty$, which implies $h(z(t))\to 0 $ as $t\to \infty$. We conclude that the accumulation point satisfies $h(\bar{z}) = 0$. Since $\calC$ is closed, the point $\bar z$ is in $\calC$. Therefore, $\bar z$ is both in $\calM$ and in the connected component of $\calC$ that contains $z(0) = x_0$.
\end{proof}

 \begin{arxiv}
 \section{Proofs of Section~\ref{sec:intro} (Introduction)}
 \label{sec:proofs_intro}
 \begin{example*}[The Stiefel  manifold]
Let $\mathcal{E} = \R^{n\times p}$. The Stiefel manifold is defined as 
\begin{equation}
\St(n,p) = \{X\in \R^{n\times p}: X\transpose X = \I_p\}. 
\end{equation}
The manifold corresponds to the defining function $h\colon \R^{n\times p} \to \mathrm{Sym}(p)\colon X \mapsto h(X) = X^\top X - \I_p$, where $\mathrm{Sym}(p)$ is the set of symmetric matrices of size $p$. For any $R<1$, all $X\in \R^{n\times p}$ such that $\norm{h(X)} \leq R$, satisfy $\sigmamin(\D h(X)) \geq 2 \sigmamin(X) \geq 2\sqrt{1-R}$. Therefore,~\aref{assu:ROI} is satisfied for any $R<1$ and $\sigmabar \leq 2\sqrt{1-R}$.
\end{example*}
\begin{proof}
First note that the set $\Sym(p)$ has dimension $p(p+1)/2$. Therefore 
\begin{align*}
\sigmamin(\D h(X)) = \sigma_{p(p+1)/2}(\D h(X)),
\end{align*}
 by definition. The differential of the defining function $h$ is given by 
\begin{equation}
\D h(X)\colon \R^{n\times p} \longrightarrow \mathrm{Sym}(p)\colon U\mapsto  \D h(X)[U] = X^\top U + U^\top X.
\end{equation} 
 To find the region where $\D h(X)$ is non-singular, we need to characterize a set of $X\in \R^{n\times p}$ such that $ X^\top U + U^\top X$ spans $\Sym(p)$. We show that this set is constituted of all the matrices of full rank. Assume $X$ has rank $p$; we can pick  $ [V, V_\perp] \in O(n)$ such that $X= VP$, for some invertible $P\in \R^{p\times p}$. Any $U\in \R^{n\times p}$ can be written as $U = VA + V_\perp B$ for some $A\in \R^{p\times p}$ and $B\in \R^{(n-p)\times p}$. We find that $\D h(X)[U] = P^\top A + A^\top P$. Therefore, $\D h(X)[U]= 0$ if and only if $A= (P^\top)^{-1}\Omega$, for some $\Omega \in \mathrm{Skew}(p)$, i.e. $p(p-1)/2$ degrees of freedom  and $\Omega^\top + \Omega = 0$. In other words, an antisymmetric $\Omega$ brings no contribution to $\D h(X)[U]$.
 
Therefore, consider $U = V(P^\top)^{-1} W$, with $W \in \Sym(p)$, this gives $\D h(X)[U] = 2W$. Hence $\D h(X)$ spans $\Sym(p)$ for any full rank $X\in \R^{n\times p}$.
By definition,
\begin{equation}
\sigmamin(\D h(X)) = \min_U \dfrac{\fronorm{\D h(X)[U]}}{\fronorm{U}}
\end{equation}
Let us express $U$ as a function of $W$. Using $X=VP$ yields
\begin{equation}
U = V(P^\top)^{-1}W = X P^{-1} (P^\top)^{-1}W = X(X^\top X)^{-1}W.
\end{equation}
This allows to write
\begin{align}
\sigmamin(\D h(X)) &= \min_{W\in \Sym(p)} \dfrac{2\fronorm{W}}{\fronorm{  X(X^\top X)^{-1}W  }}\\
&\geq  \min_{Y\in \R^{p\times p}} \dfrac{2\fronorm{Y}}{\fronorm{  X(X^\top X)^{-1}Y  }	}\\
&= \dfrac{2}{\sigmamin(X^\dagger)}\\
&= 2 \sigmamin(X).
\end{align}
Take a singular value decomposition, $X = U_1\Sigma U_2^\top$,
\begin{align}
 \norm{h(X)} &=  \fronorm{X^\top X - \I_p} \\
 &=  \fronorm{ U_1(\Sigma^\top \Sigma - \I_p)U_1^\top } \\
 &= \fronorm{\Sigma^2 - \I_p}.
\end{align}
Take $R>0$ such that $\norm{h(X)} \leq R$. This implies $ | \sigmamin(X)^2 -1| \leq R$. Firstly assume that $\sigmamin(X)^2  <1$, which gives $1- \sigmamin(X)^2 \leq R$ or $\sigmamin(X)^2 \geq 1-R$. This allows to write $\sigmamin(X) \geq \sqrt{1-R}$. Now consider the case $\sigmamin(X)^2\geq 1$, where it is clear that $\sigmamin(X) \geq \sqrt{1-R}$. In conclusion, for any $R<1$, all $X$ such that $\norm{h(X)} \leq R$, satisfy $\sigmamin(\D h(X)) \geq 2 \sigmamin(X) \geq 2\sqrt{1-R}$.
\end{proof}
\end{arxiv}

\begin{arxiv}
\section{Proofs of Section~\ref{sec:plateau-scheme} (Estimating the penalty parameter)}	
\label{appendix:plateau}
We cover in detail the question addressed in Section~\ref{sec:plateau-scheme}, that is, how does one get rid of~\aref{assu:beta_large_enough} and estimate a suitable value for the penalty parameter $\beta$ in general? Let us define the value 
\begin{equation}
B(x) := \max \left\lbrace\beta_1(x), \beta_2(x), \beta_3(x) \right\rbrace,
\label{eq:Bx}
\end{equation}
where $\beta_i(x)$ for $i=1,2,3$ are defined in Definition~\ref{def:beta123}. Ensuring that $\beta > B(x_k)$ for every iterate $x_k$ of Algorithm~\ref{algo:gradient-eigenstep} is sufficient for the algorithm to run smoothly and converge. Practically, it is possible to compute $B(x_k)$ at the current iterate while running the algorithm to increase $\beta$ if needed and ensure $\beta >B(x_k)$. However, changing $\beta$ throughout the algorithm would change the penalty function $g$, which invalidates the convergence analysis. 
 
Therefore, we propose the plateau scheme in Algorithm~\ref{algo:plateau}. It calls Algorithm~\ref{algo:gradient-eigenstep} several times, each time for a fixed number of iterations and using a constant value of $\beta$. Each call with a constant $\beta$ is called a \textit{plateau}. On each plateau, our analysis from previous sections is informative since $\beta$ is fixed, though possibly too small. With each new call, the value of $\beta$ and the length of the plateau (LP) are increased. The increase is designed to ensure that, after sufficiently many calls, $\beta$ and LP are large enough for Algorithm~\ref{algo:gradient-eigenstep} to converge and return. This stops the plateau scheme.
 
\begin{algorithm}
\caption{Plateau scheme}\label{algo:plateau}
\begin{algorithmic}[1]
\State \textbf{Given:} Functions $f$ and $h$, $0\leq \varepsilon_1 \leq R/2$, $\varepsilon_2 \geq 0$, $\gamma>1$, $x \in \calC$, $\beta_0$ and $\LP_0$.
\For{$\ell = 0, 1, 2, \dots$}
\State $x = \textrm{Gradient-Eigenstep}(x, \beta_\ell)$ with stopping criterion set to: $(B(x)\geq \beta_\ell \text{ or } k > \mathrm{LP}_\ell)$ \Comment{This is a call to Algorithm~\ref{algo:gradient-eigenstep}}
\If{ $\norm{\nabla g(x)} \leq \varepsilon_1$ \textbf{and} $\lambdamin (\nabla^2 g(x)) \geq - \varepsilon_2$ }
\State \Return
\EndIf
\If{Algorithm~\ref{algo:gradient-eigenstep} stopped because $B(x)\geq\beta_\ell$}
\State Set $B \leftarrow B(x)$
\State Set  $\LP_{\ell +1} = \! \left( \gamma B/\beta_\ell \right)^4\!\LP_\ell$ then $\beta_{\ell+1} = \gamma B$
\Else
\State $\beta_{\ell+1} = \gamma \beta_\ell$ and $\LP_{\ell+1} = \gamma^4 \LP_\ell$
\EndIf
\EndFor
\end{algorithmic}
\end{algorithm}

We proceed to show that the plateau scheme (Algorithm~\ref{algo:plateau}) converges in a finite number of iterations. Moreover, the complexity with respect to $\varepsilon_1,\varepsilon_2$ is of the same order as Algorithm~\ref{algo:gradient-eigenstep} up to logarithmic factors. Let $K(\beta)$ be the right hand side of Equation~\eqref{eq:max_iter_2nd_order} as a function of $\beta$,
\begin{equation}
K(\beta) = (g(x_0)- \underline{g}) \!\left[ \min \! \left( c_1 \tau_1 \min(\underline{\alpha_1}, \underline{t_1}) \varepsilon_1^2, c_2 \tau_2^2\min \! \left(   \min \! \left(\alpha_{02}, \dfrac{3|2c_2-1| \varepsilon_2}{M_g}\right)\!, \underline{t_2}\right)^2 \! \varepsilon_2 \right) \!\right]^{-1}\!.
\end{equation}
This gives the worst-case number of iterations required for Algorithm~\ref{algo:gradient-eigenstep} to achieve our target tolerances using a constant value of $\beta$, according to Theorem~\ref{thm:complexity}. We show that $K(\beta)$ is upper bounded by a cubic function of $\beta$ for $\beta \geq \bar{\beta}$.
\begin{lemma}\label{thm:cubic_K}
For $\varepsilon_2<1$, there exists $C_3>0 $, independent of $\varepsilon_1$, $\varepsilon_2$ and $\beta$ such that for all $\beta > \betabar$, we have 
\begin{equation}
K(\beta) \leq C_3  \max\left(\varepsilon_1^{-2}, \varepsilon_2^{-3}\right)  \beta^3,
\end{equation}
where $\betabar$ is defined in Equation~\eqref{eq:beta-critical}.
\end{lemma}
\begin{proof}
Firstly, we find the dependence on $\beta$ of the Lipschitz constants of the gradient and Hessian of $g$, namely, $L_g$ and $M_g$. From
\begin{equation}
g(x) = f(x) - \langle h(x), \lambda(x)\rangle + \beta \norm{h(x)}^2
\end{equation}
and
\begin{equation}
\nabla g(x) = \nabla f(x) - \D h(x)^*[\lambda(x)] + 2\beta \D  h(x)^*[h(x)] - \D\lambda(x)^*[h(x)],
\end{equation}
it is clear that $\nabla g(x)$ and $\nabla^2 g(x)$ are affine in $\beta$, hence $L_g(\beta)$, $M_g(\beta)$ and $\norm{\nabla g(x)}$ are affine functions of $\beta$ as well. Therefore, it is possible to find a constant $c$ such that $L_g \leq c \beta$ for all $\beta > \betabar$. We apply that reasoning to all functions affine in $\beta$, since we are interested in their behaviour for all $\beta > \betabar$. We notice that $g(x_0)$ is also an affine function of $\beta$. The value $\underline{g}$ does not depend on $\beta$.
The formula $K(\beta)$ depends on $\beta$ through the values $g(x_0), \underline{\alpha_1}, \underline{t_1}, L_g$ and $M_g$. Indeed,
\begin{align}
\underline{\alpha_1} &= \min\! \left(\alpha_{01}, \dfrac{2(1-c_1)}{L_g}\right)\\
\underline{t_1} &=\min_{x\in \calC} \min \left(\sqrt{\dfrac{R}{2C_h}}\dfrac{1}{\norm{\nabla g(x)}}    ,\dfrac{(2 \beta \sigmamin(\D h(x))^2 - \sigma_1(\D h(x)) \Clambdabar)R}{2C_h \norm{\nabla g(x)}^2}, \dfrac{1}{2\beta \opnorm{\D h(x)}^2} \right)\\ \nonumber 
\end{align}
Examination shows that $1/\underline{t_1} $  and $1/\underline{\alpha_1}$ are bounded by affine functions of $\beta$. That is, there exists constants $b_1, b_2,$ such that, for all $\beta > \betabar$, we have $1/\underline{\alpha_1} \leq b_1 \beta$ and $1/\underline{t_1} \leq b_2\beta$. Define
\begin{align*}
\underline{\alpha_2} =	\min\!\left(\alpha_{02}, \dfrac{3|2c_2-1| \varepsilon_2}{M_g}\right).
\end{align*}
Using $\varepsilon_2<1$, there exists $b_3$ such that,
\begin{align}
1/\underline{\alpha_2} &=  \max \left( 1/\alpha_{02}, \dfrac{M_g}{3|2c_2 -1|\varepsilon_2}\right)\\
&\leq \varepsilon_2^{-1} \max\left(  1/\alpha_{02},  \dfrac{M_g}{3|2c_2 -1|}\right)\\
&\leq b_3 \beta \varepsilon_2^{-1}.
\end{align}
Finally, 
\begin{align}
\underline{t_2} &= \min_{x\in \calC} \left(-\sigma_1(\D h(x)) + \sqrt{\sigma_1(\D h(x)) ^2 +2C_h R}\right)/2C_h
\end{align}
is independent of $\beta$. This implies the existence of a constant $b_4$ such that $1/\underline{t_2} \leq b_4$. In addition, for any $a, b>0$, we have $1/\min(a,b) = \max(1/a, 1/b)$. For all $\beta > \betabar$, this gives
\begin{align}
K(\beta) &=(g(x_0)- \underline{g}) \max\left( \dfrac{1}{c_1 \tau_1 \varepsilon_1^2\min(\underline{\alpha_1}, \underline{t_1})} , \dfrac{1}{c_2 \tau_2^2\varepsilon_2\min\left(   \underline{\alpha_2}, \underline{t_2}\right)^2} \right) \\
&=(g(x_0)- \underline{g}) \max\left( \dfrac{\max(1/\underline{\alpha_1}, 1/\underline{t_1})}{c_1 \tau_1 \varepsilon_1^2} , \dfrac{\max\left(1/   \underline{\alpha_2}, 1/\underline{t_2}\right)^2}{c_2 \tau_2^2\varepsilon_2} \right) \\
&\leq b_0\beta \max\left( \dfrac{\max(b_1 \beta, b_2 \beta)}{c_1 \tau_1 \varepsilon_1^2} , \dfrac{\max\left(b_3 \beta \varepsilon_2^{-1} , b_4\right)^2}{c_2 \tau_2^2\varepsilon_2} \right) \\
&\leq b_0\beta \max\left(\varepsilon_1^{-2}, \varepsilon_2^{-3}\right) \max\left( \dfrac{\max(b_1 \beta, b_2 \beta)}{c_1 \tau_1} , \dfrac{\max\left(b_3 \beta, b_4\right)^2}{c_2\tau_2^2} \right).
\end{align}
We conclude that there exists $ C_3>0 $, independent of $\varepsilon_1$, $\varepsilon_2$ and $\beta$ such that, for all $\beta \geq \betabar$, we have 
\begin{equation*}
K(\beta) \leq C_3  \max\left(\varepsilon_1^{-2}, \varepsilon_2^{-3}\right)  \beta^3. \qedhere
\end{equation*}
\end{proof}

\begin{theorem}
Under~\aref{assu:ROI},~\aref{assu:boundedMandC},~\aref{assu:lipschitzh} and \aref{assu:x0}, the plateau scheme of Algorithm~\ref{algo:plateau} returns an $(\varepsilon_1,2 \varepsilon_1, \varepsilon_2 + C\varepsilon_1)$-SOSP in at most $\bigO\left(\max\left\lbrace \varepsilon_1^{-2}, \varepsilon_2^{-3}\right\rbrace   \max\left\lbrace \log_\gamma \varepsilon_1^{-2}, \log_\gamma \varepsilon_2^{-3}\right\rbrace \right)$ iterations of Algorithm~\ref{algo:gradient-eigenstep}, where $C$ is defined in Corollary~\ref{corollary:approx_stationarity_hessian-global}.
\end{theorem}
\begin{proof}
For $C_3$ provided by Theorem~\ref{thm:cubic_K}, define the value  $C_2 = C_3  \max\left(\varepsilon_1^{-2}, \varepsilon_2^{-3}\right)$ which is independent of $\beta$. Algorithm~\ref{algo:plateau} is guaranteed to converge when $\beta_\ell > \betabar$ and $\LP_\ell \geq K(\beta_\ell)$. That is, when $\beta_\ell$ is large enough to satisfy~\aref{assu:beta_large_enough} and the length of the plateau is large enough to allow Algorithm~\ref{algo:gradient-eigenstep} to converge in a worst-case number of iterations. From Algorithm~\ref{algo:plateau}, it is clear that $\beta_\ell \geq \beta_0 \gamma^\ell$. Therefore, if $\ell >  \lceil \log_{\gamma}(\betabar/\beta_0) \rceil$, it follows that 
\begin{align}\label{eq:beta+bound}
\beta_\ell \geq \beta_0 \gamma^\ell > \beta_0 \gamma^{ \log_{\gamma}(\betabar/\beta_0)} = \betabar,
\end{align}
which indicates for which $\ell $ large enough the condition $\beta_\ell > \betabar$ is met. Regarding the length of the plateaus, it is clear that $\LP_\ell \geq \LP_0 \gamma^{4\ell}$. One can also infer from Algorithm~\ref{algo:plateau} that $\beta_\ell \leq \max(\beta_0, \betabar)\gamma^\ell$. Indeed if $\beta_0 > \betabar$, then $\beta_\ell = \beta_0 \gamma^\ell$ since lines 8 and 9 of Algorithm~\ref{algo:plateau} are not executed. When $\betabar \geq \beta_0$, $\beta_\ell \leq \betabar \gamma^\ell$. Using that $K(\beta) \leq C_2 \beta^3 $ for all $\beta>\betabar$, we enforce
\begin{equation}
\LP_0 \gamma^{4\ell} \geq C_2\left(\max(\beta_0, \betabar)\gamma^\ell\right)^3,
\label{eq:I-want-gamma}
\end{equation}
from which it follows that
\begin{equation}
\LP_\ell \geq \LP_0 \gamma^{4\ell} \geq C_2 \left(\max(\beta_0, \betabar)\gamma^\ell\right)^3 \geq C_2 \beta_\ell^3 \geq K(\beta_\ell).
\end{equation}
 Equation~\eqref{eq:I-want-gamma} simplifies to
\begin{equation}
\gamma^\ell \geq \dfrac{C_2 \max(\beta_0, \betabar)^3}{LP_0}
\end{equation}
or
\begin{equation}
\ell \geq \left\lceil \log_\gamma \left( \dfrac{C_2 \max(\beta_0, \betabar)^3}{\LP_0} \right) \right\rceil.
\end{equation}
In conclusion, the maximum number of plateaus is
\begin{equation}
\ell^* := \max\left( \left\lceil \log_\gamma \left( \dfrac{C_2 \max(\beta_0, \betabar)^3}{LP_0} \right) \right\rceil, \lceil \log_{\gamma}(\betabar/\beta_0) \rceil +1 \right).
\end{equation}
The maximum value of $\beta$ is
\begin{equation}
\beta^* =  \max(\beta_0, \betabar) \gamma^{\ell^*}.
\end{equation}
On any plateau the maximum number of iterations of Gradient-Eigenstep is $K(\beta^*)$, because for $\beta_\ell = \beta^*$, the plateau is long enough to allow convergence and $\beta^*\geq \betabar$. The worst-case number of Gradient-Eigenstep iterations is at most $K(\beta^*)\ell^*$ with
\begin{equation*}
 K(\beta^*)\ell^* \leq C_3 \max\left(\dfrac{1}{\varepsilon_1^2}, \dfrac{1}{\varepsilon_2^3}\right)  (\beta^*)^3  \max\left( \left\lceil \log_\gamma \left( \dfrac{C_3 \max\left(1/\varepsilon_1^2, 1/\varepsilon_2^3\right)  \max(\beta_0, \betabar)^3}{LP_0} \right) \right\rceil, \lceil \log_{\gamma}(\betabar/\beta_0) \rceil +1 \right).\qedhere
\end{equation*}
\end{proof}
\end{arxiv}

\begin{arxiv}
\section{Proofs of Propositions~\ref{prop:bertsekas} and~\ref{prop:strict-2nd-order-point}}
\label{sec:appendix-bertsekas}
\begin{proposition}[\citep{bertsekas1982constrained},~Prop.~4.22] 
	Let $g(x) = \Lcal_\beta(x, \lambda(x))$ be Fletcher's augmented Lagrangian and assume $\M \subset \Dcal$, where $\Dcal = \{ x \in \Ecal : \rank( \D h(x)) = m \}$ and $\M = \left\{ x\in \calE: h(x) = 0\right\}$. 
	\begin{enumerate}
		\item For any $\beta$, if $x$ is a first-order critical point of~\eqref{eq:P}, then $x$ is a first-order critical point of $g$.
		\item Let $x \in \Dcal$ and $\beta > \beta_1(x)$. If $x$ is a first-order critical point of $g$, then  $x$ is a first-order critical point of~\eqref{eq:P}.
		\item Let $x$ be a first-order critical point of~\eqref{eq:P} and let $K$ be a compact set. Assume $x$ is the unique global minimum of $f$ over $\M \cap K$ and that $x$ is in the interior of $K$. Then, there exists $\beta$ large enough such that $x$ is the unique global minimum of $g$ over $K$.
		\item Let $x\in \calD$ and $\beta > \beta_1(x)$. If $x$ is a local minimum of $g$, then $x$ is a local minimum of~\eqref{eq:P}.
	\end{enumerate}
\end{proposition}
\begin{proof}
\begin{enumerate}
\item If $x$ is a first-order critical point for~\eqref{eq:P}, then $h(x)=0$ and $\nabla f(x) = \D h(x)^*[\lambda(x)]$ by definition. Therefore,
 \begin{equation}
\nabla g(x) = \nabla f(x) - \D h(x)^*[\lambda(x)] + 2\beta \D h(x)^*[h(x)] - \D\lambda(x)^*[h(x)] = 0.
\label{eq:nabla_g_zero}
\end{equation}
\item  Take $x\in \Dcal$ with $\nabla g(x)= 0$. We find that
\begin{align}
0 &= \D h(x) \!\left[\nabla g(x)\right]\\
   &= \D h(x) \!\left[ \nabla \!\left( x \mapsto f(x) - \langle h(x) , \lambda(x)\rangle + \beta\norm{h(x)}^2\right)\! (x)\right]\\
	&= \D h(x)\!\left[ \nabla f(x) - \D h(x)^*[\lambda(x)] + 2\beta \D h(x)^*[h(x)] - \D\lambda(x)^*[h(x)] \right]\\
	&=  \D h(x) \!\left[ \grad_{\M_x} f(x) \right] + 2\beta \D h(x) \left[ \D h(x)^*[h(x)] \right]-  \D h(x)\!\left[\D\lambda(x)^* [h(x)]\right]\\
	&= \left\lbrace 2\beta \D h(x) \D h(x)^* -  \D h(x)\D\lambda(x)^* \right\rbrace h(x), \label{eq:curvy}
\end{align}
where the term $\D h(x)\!\left[ \grad_{\M_x} f(x) \right]$ vanishes because the Riemannian gradient of $f$ restricted to $\calM_x$ is tangent to $\calM_x$ at $x$, and by definition this tangent space is the kernel of $\D h(x)$. Using Weyl's inequality (Equation~\ref{eq:weyl}) gives
\begin{align}
\sigmamin\left( 2\beta \D h(x) \D h(x)^* -  \D h(x)\D\lambda(x)^* \right) &\geq \sigmamin(2\beta \D h(x) \D h(x)^* ) - \sigma_1( \D h(x)\D\lambda(x)^*)\\
&\geq 2\beta \sigmamin^2(\D h(x)) - \sigma_1(\D h(x))C_\lambda(x).
\label{eq:sigma_min_curvy}
\end{align}
If $\beta > \beta_1(x)$ (Definition~\ref{def:beta123}), we see from~\eqref{eq:sigma_min_curvy} that the linear operator that appears on the right hand side of~\eqref{eq:curvy} is nonsingular, and therefore~\eqref{eq:curvy} implies $h(x) = 0$. Going back to~\eqref{eq:nabla_g_zero} and using $\nabla g(x) = 0$ together with $h(x) = 0$, it follows that $ \nabla f(x) = \D h(x)^*[\lambda(x)]$. This proves that $x$ is a first-order critical point of~\eqref{eq:P} with multipliers $\lambda(x)$.
	\item The set $K$ is compact and therefore, for any $\beta$, there exists a global minimizer of $g$ inside $K$. We proceed by contradiction. Therefore, for any integer $k>0$, there exists $\beta_k\geq k$ and a global minimizer $x_k$ of $g$ over $K$ such that $x_k \neq x$. This implies
\begin{equation}
g(x_k) \leq g(x) = f(x).
\label{eq:contradiction}
\end{equation}
Hence, $\limsup_{k\longrightarrow\infty} g(x_k) \leq f(x)$. We shall show that $\left(x_k\right)_{k\in \mathbb{N}} \longrightarrow x$. Let $\bar{x}$ be a limit point of $\left(x_k\right)_{k\in \mathbb{N}} $. Since $\beta_k\longrightarrow \infty$, we have $h(\bar{x}) = 0$. Therefore,
\begin{equation}
f(\bar{x}) = g(\bar{x}) \leq f(x).
\end{equation}
Given that $x_k\in K$ for all $k$, compactness ensures that $\bar{x}\in K$. Since $x$ is the unique global minimizer of $f$ over $K\cap \Mcal$, it follows that $\bar{x}= x$.

Now we show that there exists some index $k$ such that $x_k = x$. Since $x_k \longrightarrow x$, we can take an open ball $B$ centered around $x$ such that $x_k\in B$ for $k$ sufficiently large. We also chose $B$ such that $\mathrm{cl}(B)\subset K$. From item 2, for all $\beta\geq \bar{\beta}_1$, every critical point of $g$ inside $B$ is a first-order critical point of~\eqref{eq:P}. Hence, for all $k$ sufficiently large, $x_k$ is a first-order critical point of $f$ on $\M$ (as it is a global min of $g$ inside $B\subset K$, it is a critical point of $g$ and point 2 applies on the compact set $\mathrm{cl}(B)$). This implies $h(x_k) =0$ and $f(x_k) \leq f(x)$ from~\eqref{eq:contradiction}. Since $x$ is the unique global minimizer of $f$ over $K \cap \Mcal$, it follows that $x_k = x $ for all $k$ sufficiently large. This contradicts that $x_k \neq x$ for all $k$ and proves the original statement.
\item From item 2, for all $\beta > \beta_1(x)$, if $x$ is a local minimum of $g$, then $x$ is a first-order critical point of~\eqref{eq:P}. This implies \begin{equation}
f(x) = g(x)
\end{equation}
since $h(x) =0$. From the local optimality of $x$ for $g$, there exists a ball $B\subset \calE$ centered at $x$ such that
\begin{align*}
g(x) &\leq g(y)  && \textrm{ for all } y\in B.
\end{align*}
This holds a fortiori for all $y \in B \cap \Mcal.$
Combining the last two results gives
\begin{align}
f(x) &= g(x) \leq g(y) = f(y)  && \textrm{ for all }y\in B \cap \Mcal
\label{eq:x_local_min_B}
\end{align}
where the equalities hold because $x, y \in \calM$ imply $h(x) = h(y) = 0$, and the inequality holds owing to $x, y \in B$. Equation~\eqref{eq:x_local_min_B} implies that $x$ is a local minimizer of~\eqref{eq:P}.\qedhere 
\end{enumerate}
\end{proof}
\end{arxiv}

\begin{arxiv}
\begin{proposition}[\citep{fletcher1970class}]
If $x\in \M$ is a local minimizer of~\eqref{eq:P} with $\Hess_\M f(x) \succ 0$, there exists $\beta$ large enough such that $\nabla^2 g(x) \succ 0$. 
\end{proposition}
\begin{proof}
\label{sec:appendix-fletcher}
Starting from the expression of the Hessian of $g$ at feasible points~\eqref{eq:hessian_g_feasible}, split $\dot x = u + v \calE$ into a tangent part ($u$) and a normal part ($v$). Define the following symmetric operator on $\calE$:
\begin{align*}
	H & = \nabla^2 f(x) - \sum_{i = 1}^{m} \lambda_i(x) \nabla^2 h_i(x).
\end{align*}
Then,
\begin{align*}
	\inner{\dot x}{\nabla^2 g(x)[\dot x]} & = \inner{u}{Hu} + 2\inner{v}{Hu} + \inner{v}{Hv} \\
				& \quad + 2\beta \|\D h(x)[v]\|^2 \\
				& \quad - 2\inner{\D\lambda(x)[u + v]}{\D h(x)[v]}.
\end{align*}
Note that $\inner{u}{Hu} = \inner{u}{\Hess_\calM f(x)[u]}$.
Also, let $\sigmamin > 0$ denote the $m$th singular value of $\D h(x)$.
Then,
\begin{align*}
	\inner{\dot x}{\nabla^2 g(x)[\dot x]} & \geq \lambdamin(\Hess_\calM f(x)) \|u\|^2 - 2\|H\|\|u\|\|v\| \\
	& \quad + (2\beta \sigmamin^2 + \lambdamin(H)) \|v\|^2 \\
	& \quad - 2 \|\D\lambda(x)\| \|\D h(x)\| \|u+v\| \|v\| \\
	& \geq \lambdamin(\Hess_\calM f(x)) \|u\|^2 - 2 \left( \|H\| + \|\D\lambda(x)\| \|\D h(x)\| \right) \|u\|\|v\| \\
	& \quad + \left(2\beta \sigmamin^2 + \lambdamin(H) - 2 \|\D\lambda(x)\| \|\D h(x)\|\right) \|v\|^2 \\
	& = \begin{bmatrix}
		\|u\| & \|v\|
	\end{bmatrix} \begin{bmatrix}
		\lambdamin(\Hess_\calM f(x)) & -(\|H\| + \|\D\lambda(x)\| \|\D h(x)\|) \\ -(\|H\| + \|\D\lambda(x)\| \|\D h(x)\|) & 2\beta \sigmamin^2 + \lambdamin(H) - 2 \|\D\lambda(x)\| \|\D h(x)\|
	\end{bmatrix}
	\begin{bmatrix}
	\|u\| \\ \|v\|
	\end{bmatrix}.
\end{align*}
If $\Hess_\calM f(x) \succ 0$, we can pick $\beta>0$ large enough such that the $2\times 2$ matrix above has a positive determinant and trace; and is therefore positive definite. Note however that if we only have $\Hess_\calM f(x) \succeq 0$, then it is not possible (in general) to make the $2\times 2$ matrix above positive semidefinite with any finite $\beta$.
\end{proof} 
\end{arxiv}

\end{document}